\documentclass[11pt]{article}

\usepackage[margin=1.15in]{geometry}
\usepackage[T1]{fontenc}
\usepackage[utf8]{inputenc}
\usepackage{lmodern}
\usepackage{enumitem}
\usepackage{hyperref}
\usepackage{cmbright}
\usepackage{amssymb,amsmath,latexsym,amsthm,mathtools,graphicx,bbm,mathptmx,cite,ifthen,color}
\usepackage{authblk}

\newcommand\commentout[1]{}

\newtheorem{theorem}{Theorem}[section]
\newtheorem{proposition}[theorem]{Proposition}
\newtheorem{lemma}[theorem]{Lemma}
\newtheorem{corollary}[theorem]{Corollary}
\newtheorem{problem}[theorem]{Problem}
\newtheorem{example}[theorem]{Example}
\theoremstyle{definition}
\newtheorem{definition}[theorem]{Definition}
\theoremstyle{remark}
\newtheorem{remark}[theorem]{Remark}

\title{Optimal \L{}ojasiewicz–Simon Exponents and the Intrinsic Singularities of Analytic Functionals}

\author[1]{Tewodros Amdeberhan}
\author[1]{T\`ai Huy H\`a}

\affil[1]{Tulane University, Department of Mathematics,} 
\affil[ ]{6823 St. Charles Avenue, New Orleans LA 70118, USA}
\affil[ ]{\texttt{tamdeber@tulane.edu} and \texttt{tha@tulane.edu}}

\date{}

\begin{document}
\maketitle

\begin{abstract} 
We prove that the optimal \L{}ojasiewicz--Simon exponent of a real-analytic functional on a Banach space is an intrinsic invariant of its finite-dimensional Lyapunov--Schmidt singularity under natural Fredholm and pairing-compatible hypotheses. More precisely, Lyapunov--Schmidt reduction preserves the entire set of admissible exponents, while the reduced germ on the kernel of the Hessian is intrinsic up to analytic right-equivalence tangent to the identity. We then identify this invariant with a real relative Łojasiewicz exponent of the pair consisting of the Jacobian ideal of the reduced energy and its principal energy ideal. Consequently, the optimal exponent admits equivalent descriptions by analytic arcs and by a finite maximum of divisorial ratios; in particular, it is rational. In Newton-nondegenerate cases, these formulas yield explicit monomial expressions.

We apply the theory to several classes of singular energies. For the sum $Q_k$ of the squares of all $k\times k$ minors of a real matrix, we determine the associated ideal pair up to real integral closure and obtain the exact local exponent
$$1-\frac1{2(k-s)}$$
at matrices of rank $s<k$. For Yang--Mills energy, we prove that the product flat connection on $\mathbb T^2$ has optimal exponent $3/4$ for every nonabelian compact structure group. For $\mathrm{SU}(2)$, the determinantal reduction extends this value to product flat connections on every $\mathbb T^d$, while sufficiently small nonzero constant flat connections have exponent $1/2$. Finally, for resonant semilinear Dirichlet energies, we identify higher-order resonant obstructions, establish an all-order obstruction ladder for simple resonance, and obtain a complete parity-dependent classification on $(0,\pi)$.
\end{abstract}

\section{Introduction}

The \L{}ojasiewicz--Simon gradient inequality is one of the principal mechanisms by which finite-dimensional real-analytic geometry enters nonlinear analysis. Simon's foundational work \cite{Simon1983} reduces the behavior near a degenerate critical point to a finite-dimensional analytic problem, and subsequent formulations by Chill and Feehan--Maridakis substantially clarified the functional-analytic hypotheses and the applications to gradient-like evolution equations \cite{Chill2003,FeehanMaridakis2020,FeehanMaridakisMemoirs}. In the Morse--Bott case the optimal exponent is $1/2$, a fact developed in particular by Feehan \cite{Feehan2020,Feehan2022}. For genuinely degenerate critical points, however, the numerical value of the optimal exponent records finer singularity data that are not visible from Fredholm theory alone.

Finite-dimensional reduction is already central in the analysis of slowly converging geometric flows. The first nonzero homogeneous obstruction appears in the work of Adams--Simon and in the Lyapunov--Schmidt analysis of slowly converging Yamabe flows by Carlotto--Chodosh--Rubinstein \cite{AdamsSimon1988,CCR2015}. More recent work of Choi--Hung studies rates and asymptotic directions for nonlinear evolution equations \cite{ChoiHung2023,ChoiHung2024}. These works make clear that the reduced finite-dimensional singularity is not merely a technical device. Our purpose is to isolate the optimal \L{}ojasiewicz--Simon exponent itself as an intrinsic singularity-theoretic invariant and to bring to bear on it the classical theory of integral closure, analytic arcs, divisorial valuations, and Newton polyhedra.

There are two points of contact with earlier abstract splitting results that are important for positioning the paper accurately. Feehan proves that adding a nondegenerate quadratic factor preserves a \L{}ojasiewicz exponent and gives a Banach splitting theorem for degenerate critical points \cite[Lemma~3.3 and Theorem~5]{Feehan2020}. Thus, in his setting, the exponent of the full germ is already identified with that of a finite-dimensional residual germ. Our first contribution is a direct exact-reduction theorem in an asymmetric Banach pair $X$--$Y$, where the derivative of the energy is represented by a map $M:X\to Y$ rather than by a gradient in $X$ or $X^*$. More precisely, Theorem~\ref{thm:exact_reduction} identifies the entire sets of admissible exponents, not only their infima. Our second contribution is an intrinsicity theorem adapted to the canonical kernel $K=\ker L$: Theorem~\ref{thm:right-equivalence} shows that two compatible reductions are related by an analytic right-equivalence whose differential at the origin is the identity. Consequently the first nonzero homogeneous obstruction is literally the same polynomial on the fixed vector space $K$.

The algebraic object selected by Simon's inequality is the pair
$$
\mathcal P(\Gamma)=\bigl(J(\Gamma),(\Gamma)\bigr),
$$
where $\Gamma$ is a reduced energy and $J(\Gamma)$ its Jacobian ideal. We call this the \emph{Simon pair}. The terminology is new, but the finite-dimensional ingredients are classical. Integral closure and the valuative theory of \L{}ojasiewicz exponents go back to the work of Lejeune-Jalabert--Teissier and Risler, and the real arc criterion used here is developed by Gaffney--Ara\'ujo dos Santos \cite{LejeuneTeissier2008,GaffneyAraujo2010}. Newton-polyhedral computations of relative exponents and Newton-nondegenerate Jacobian ideals are studied by Bivi\`a-Ausina \cite{Bivia2003,Bivia2005}. Theorem~\ref{thm:Simon-pair-formula} transfers this classical finite-dimensional machinery to the intrinsic reduced singularity of the Banach functional and gives
$$
\theta_{\mathrm{LS}}(E,0)
=\sup_\gamma\frac{\nu_\gamma(J(\Gamma))}{\nu_\gamma(\Gamma)}
=\max_i\frac{\operatorname{ord}_{F_i}J(\Gamma)}{\operatorname{ord}_{F_i}\Gamma}.
$$
In particular, the optimal Banach exponent is rational and is attained by a divisorial order. The elementary order estimate $\theta_{\mathrm{LS}}\ge(m-1)/m$, where $m=\operatorname{ord}_0\Gamma$, is only the first approximation: singular directions of the leading homogeneous term may force a much larger exponent, as the example $x^3+y^{100}$ shows. We also obtain a closed determinantal family in which the full valuative problem can be solved. If $Q_k$ denotes the sum of the squares of all $k\times k$ minors of a real matrix and $I_j$ is the ideal of $j\times j$ minors, Theorem~\ref{thm:determinantal-Simon} proves
$$
\overline{(Q_k)}^{\,\mathbb R}=\overline{I_k^2}^{\,\mathbb R},
\qquad
\overline{J(Q_k)}^{\,\mathbb R}=\overline{I_{k-1}I_k}^{\,\mathbb R},
$$
and gives the exact local gradient exponent $1-1/(2(k-s))$ at the rank-$s$ stratum. This determinantal calculation is subsequently used in the higher-dimensional Yang--Mills application.

Our first application concerns Yang--Mills energy at product flat connections on tori. Feehan's work emphasizes that the product $\mathrm{SU}(2)$ connection on $\mathbb T^2$ is a basic non-Morse--Bott singular point of the flat-connection moduli space \cite{FeehanFlat2019}. In the gauge-fixed $W^{1,p}$--$W^{-1,p}$ model used below, the Lyapunov--Schmidt correction vanishes exactly on $\mathbb T^2$ for every compact Lie group $G$, and the reduced energy is $\Gamma_{\mathrm{gf}}(\xi,\eta)=\kappa\|[\xi,\eta]\|^2$. Fisher's $3/4$ gradient inequality for norm-squares of linear moment maps \cite[Theorem~4.7]{Fisher2014}, together with the quartic order obstruction, yields the exact exponent $3/4$ whenever $\mathfrak g$ is nonabelian; for abelian $\mathfrak g$ the exponent is $1/2$. For $G=\mathrm{SU}(2)$ the constant-mode reduction persists on every $\mathbb T^d$. There the reduced energy is, up to scale, the sum of the squares of all $2\times2$ minors of the $3\times d$ matrix of harmonic coefficients. Theorem~\ref{thm:YM-higher-torus} therefore gives exponent $3/4$ at the product connection and $1/2$ at every sufficiently small nonzero constant flat connection.

Our second application is the resonant semilinear Dirichlet energy
$$
\mathcal E_N(u)=\frac12\int_\Omega(|\nabla u|^2-\lambda u^2)\,dx+\frac1N\int_\Omega u^N\,dx.
$$
The even-power exponent $(N-1)/N$ lies within the earlier semilinear theory of Haraux--Jendoubi--Kavian \cite[Theorem~2.2]{HarauxJendoubiKavian2003}; their Theorem~2.5 also treats a simple-kernel cubic noncancellation condition. Our reduction gives more precise information about successive resonant obstructions. When the leading polynomial $P_N(\xi)=N^{-1}\int_\Omega\xi^Ndx$ vanishes identically on the resonant eigenspace, Theorem~\ref{thm:semilinear_exp} identifies the next homogeneous obstruction and gives the sharp exponent $(2N-3)/(2N-2)$ whenever that obstruction has an isolated critical point, without assuming that the eigenvalue is simple. For a simple resonance, Theorem~\ref{thm:obstruction-ladder} gives the exact factorization $\Gamma(te)=t^2H(t^{N-2})$ and hence a discrete all-order ladder of possible exponents. Finally, on $(0,\pi)$, Theorem~\ref{thm:interval-classification} determines the exponent for every $N$ and every resonant eigenvalue: for odd $N$ the nominal exponent occurs at odd eigenmodes, while even eigenmodes exhibit a systematic cancellation and have exponent $(2N-3)/(2N-2)$.

These two applications are deliberately complementary. In the Yang--Mills problem, the reduction is exact on the harmonic modes and the resulting finite-dimensional singularity is algebraic and, for $SU(2)$, determinantal. In the semilinear problem, by contrast, the transverse correction is generally nontrivial, and successive cancellations reveal higher-order resonant obstructions. Together they illustrate two complementary roles of the framework: it can identify a reduced singularity that is already geometrically visible, and it can uncover singularity-theoretic information that is hidden in the original infinite-dimensional formulation.

The final part of Section~\ref{sec:intrinsic} records the usual implications of the optimal exponent for gradient-like evolution equations and formulates a valuative realization problem: a divisor computing the algebraic exponent need not automatically be represented by an actual gradient trajectory. We keep this as a problem rather than a theorem, since the algebraic finite-max formula and the dynamical realization question require different information.

The paper is organized as follows. Section~\ref{sec:framework} develops the asymmetric Lyapunov--Schmidt framework, exact reduction, the order bound, and the Morse--Bott dichotomy. Section~\ref{sec:intrinsic} proves intrinsicity of the reduced singularity and develops the Simon pair, its real valuative and Newton formulas, and the determinantal family above. Section~\ref{sec:YM} treats Yang--Mills energy, first on $\mathbb T^2$ for a general compact structure group and then on $\mathbb T^d$ for $\mathrm{SU}(2)$. Section~\ref{sec:semilinear} develops the higher-order resonant obstructions, the simple-resonance obstruction ladder, and the complete interval classification.

\medskip

\noindent\textbf{Acknowledgment.} The authors thank Bernard Teissier and Huy Vui H\`a for suggesting that we study the \L{}ojasiewicz--Simon exponent using the algebraic framework for the \L{}ojasiewicz exponent developed in \cite{Ha2026}. The second author is partially supported by a Simons Foundation grant.

\section{Framework: Asymmetric Lyapunov--Schmidt reduction}\label{sec:framework}

Our first task is to isolate precisely which part of the Łojasiewicz--Simon inequality is genuinely finite dimensional. We work with an asymmetric Banach pair because in the applications the derivative of the energy is naturally represented by a map into a target space $Y$, rather than by a gradient taking values in $X$ itself. The compatible splitting below separates the kernel directions, where the singular behavior survives, from the transverse directions, where the Hessian is invertible.

Throughout this section, $X$ and $Y$ are real Banach spaces equipped with a continuous bilinear pairing $\langle\cdot,\cdot\rangle_{Y,X}:Y\times X\to\mathbb R$, and $U\subseteq X$ is an open neighborhood of $0$. Let $E:U\to\mathbb R$ be real analytic, with $E(0)=0$ and $dE(0)=0$. Assume there is a real analytic map $M:U\to Y$ representing the derivative, so that $dE(x)[v]=\langle M(x),v\rangle_{Y,X}$ for $x\in U$ and $v\in X$. Write $L:=DM(0):X\to Y$ for the Hessian.

\smallskip
Assume that $L$ is Fredholm. The symmetry of $L$ with respect to the pairing is automatic. Indeed, for $x,y\in X$, we have
$D^2E(0)[x,y]=\langle Lx,y\rangle_{Y,X}.$ 
Since $E$ is real analytic, it is $C^2$, and hence $D^2E(0)$ is symmetric;  thus moreover
$$\langle Lx,y\rangle_{Y,X}=D^2E(0)[x,y]=D^2E(0)[y,x]=\langle Ly,x\rangle_{Y,X}.$$

Let $K:=\ker L$. Choose a closed complement $Z\subset X$ such that $X=K\oplus Z$.
Assume, in addition, that there is a compatible topological splitting
$$Y=\operatorname{Ran}L\oplus Y_K \qquad \text{where} \qquad
Y_K=Z^\circ :=\{y\in Y:\langle y,z\rangle_{Y,X}=0 \text{ for every }z\in Z\}, $$
and that the restriction map
$\rho:Y_K\longrightarrow K^*$ such that $\rho(y)=\langle y,\cdot\rangle_{Y,X}\big|_K$,
is an isomorphism.
These compatibility assumptions imply that $L$ has index zero. Indeed,
since $Y = \operatorname{Ran} L \oplus Y_K$ and $Y_K \cong K^*$, we have
$\operatorname{codim} \operatorname{Ran} L = \dim Y_K = \dim K^* = \dim K$.
Thus $\operatorname{ind}L=\dim\ker L-\operatorname{codim}\operatorname{Ran}L=0$. This establishes that $L$ has index zero.

\subsection{The \L{}ojasiewicz--Simon exponent and the compatible splitting}

We define the \L{}ojasiewicz--Simon exponent of $E$ at $0$ by
$$\theta_{\mathrm{LS}}(E,0) := \inf \left\{ \theta \in [1/2,1) : \|M(x)\|_Y \ge C |E(x)|^{\theta} \, \text{ for all } x \text{ sufficiently close to } 0 \right\}.$$
Since $E(0)=0$, for $x$ sufficiently close to $0$ we have $|E(x)| < 1$. Consequently, if the inequality holds for some exponent $\theta$, it trivially holds for any $\theta' \in (\theta, 1)$. The infimum therefore represents the smallest, or sharpest, admissible exponent. A priori, the set of admissible exponents could be empty, but the classical finite-dimensional real analytic \L{}ojasiewicz gradient inequality---transferred via the exact reduction below---ensures it is non-empty for real-analytic functionals. We will refer to $\theta_{\mathrm{LS}}(E,0)$ as the \emph{optimal} exponent once this set is established to be non-empty.

\smallskip
The use of the target space norm $\|\cdot\|_Y$ is intrinsic to the asymmetric Banach formulation: the gradient map $M(x)$ is evaluated natively in $Y$, bypassing potential loss of closed range in the abstract dual space $X^*$.

\begin{lemma}\label{lem:compatibility-consequences}
Under the compatibility hypotheses above, $\operatorname{Ran}L=K^\circ$, the pairing separates points of $Y$, and the restriction $L|_Z:Z\to\operatorname{Ran}L$ is an isomorphism. In particular, $dE(x)=0$ if and only if $M(x)=0$ for $x$ near $0$.
\end{lemma}

\begin{proof}
Since $L$ is Fredholm of index zero, $K$ is finite dimensional and $\dim K=\operatorname{codim}\operatorname{Ran}L$. By the symmetry of $L$, we have $\operatorname{Ran}L\subseteq K^\circ$, where $K^\circ:=\{y\in Y:\langle y,k\rangle_{Y,X}=0\text{ for every }k\in K\}$.
Because $\rho:Y_K\to K^*$ is an isomorphism, $Y_K\cap K^\circ=\{0\}$. If $y\in K^\circ$, write $y=y_1+y_2$ with $y_1\in\operatorname{Ran}L$ and $y_2\in Y_K$. Since $\operatorname{Ran}L\subseteq K^\circ$, also $y_2=y-y_1\in K^\circ$, and hence $y_2=0$. Thus $y\in\operatorname{Ran}L$, proving $\operatorname{Ran}L=K^\circ$.

If $y\in Y$ annihilates all of $X$, then it annihilates $Z$, so $y\in Z^\circ=Y_K$; it also annihilates $K$, so $\rho(y)=0$ and therefore $y=0$. Thus the pairing separates points of $Y$. Consequently the map $Y\to X^*$ induced by the pairing is injective, and $dE(x)=0$ is equivalent to $M(x)=0$.

Finally, $L|_Z$ is injective because $Z\cap\ker L=\{0\}$, and it is onto $\operatorname{Ran}L$ because $X=K\oplus Z$ and $L(K)=0$. The bounded inverse theorem gives the claimed isomorphism.
\end{proof}

Let $R \colon Y \to \operatorname{Ran} L$ be the continuous projection associated with $Y = \operatorname{Ran} L \oplus Y_K$.

\subsection{Reduction and Comparison Estimates}

Define $F$ on a neighborhood of $(0,0)$ in $K\times Z$ by $F(\xi,\eta):=RM(\xi+\eta)$. Since $D_{\eta}F(0,0)=L|_Z$ is an isomorphism, the analytic implicit function theorem provides neighborhoods $V\subset K$ and $W\subset Z$ of $0$ and a unique real analytic map $\psi:V\to W$ such that $\psi(0)=0$ and $RM(\xi+\psi(\xi))=0$. We refer the latter as the \emph{Lyapunov--Schmidt reduction} equation. Differentiating this reveals $D\psi(0)=0$. 

\smallskip
Define the Lyapunov--Schmidt parametrization $c(\xi):=\xi+\psi(\xi)$ and the real analytic germ $\Gamma:V\to\mathbb R$ by $\Gamma(\xi):=E(c(\xi))$; we call $\Gamma$ the \emph{reduced energy} attached to the reduction. After shrinking the neighborhoods, every $x$ sufficiently close to $0$ has a unique representation $x=c(\xi)+z$ with $\xi\in V$ and $z\in Z$: if $x=\xi+\eta$ is its $K\oplus Z$ decomposition, then $z=\eta-\psi(\xi)$.

The Lyapunov--Schmidt reduction used below is standard in proofs of the \L{}ojasiewicz--Simon inequality; see, for example, Simon~\cite{Simon1983} and Feehan--Maridakis~\cite{FeehanMaridakis2020}. We record the resulting identity in the asymmetric pairing used here.

\begin{lemma} \label{lem:gradient_identity}
For every $\xi \in V$, we have $D\Gamma(\xi) = \rho\bigl(M(c(\xi))\bigr) \in K^*$.
Equivalently, for every $h \in K$, 
$$D\Gamma(\xi)[h] = \langle M(c(\xi)),h\rangle_{Y,X}.$$
\end{lemma}

\begin{proof}
By the chain rule,
$$D\Gamma(\xi)[h] = dE(c(\xi)) \bigl[ Dc(\xi)h \bigr] = \langle M(c(\xi)), h + D\psi(\xi)h \rangle_{Y,X}.$$
By the defining equation of the Lyapunov--Schmidt reduction, we have $R M(c(\xi)) = 0$. Therefore, $M(c(\xi))$ lies in the kernel of the projection $R$, which is exactly $Y_K$. Since we chose the compatible topological splitting such that $Y_K = Z^{\circ}$, the gradient $M(c(\xi))$ annihilates any vector in $Z$. 

\smallskip
Because $D\psi(\xi)h\in Z$, one has $\langle M(c(\xi)),D\psi(\xi)h\rangle_{Y,X}=0$. Hence $D\Gamma(\xi)[h]=\langle M(c(\xi)),h\rangle_{Y,X}$ for every $h\in K$, which is exactly $D\Gamma(\xi)=\rho(M(c(\xi)))$.

\smallskip
Furthermore, because $\rho \colon Y_K \to K^*$ is an isomorphism between finite-dimensional spaces and $M(c(\xi)) \in Y_K$, their norms are uniformly equivalent. Consequently, there exists $C \ge 1$ such that
$$C^{-1}\|M(c(\xi))\|_Y \le \|D\Gamma(\xi)\|_{K^*} \le C\|M(c(\xi))\|_Y, $$
uniformly for $\xi$ sufficiently close to $0$. Thus, the finite-dimensional gradient norm perfectly controls the target-space norm of the full infinite-dimensional gradient map along the reduction manifold.
\end{proof}

Estimates of this type are standard consequences of the Lyapunov--Schmidt reduction used in proofs of the \L{}ojasiewicz--Simon inequality; see, for example, \cite{Simon1983,FeehanMaridakis2020}. We state them here in the X--Y norms of the asymmetric framework. 

\begin{lemma} \label{lem:comparison}
After shrinking the neighborhoods if necessary, there exists $C \ge 1$ such that whenever $x=c(\xi)+z$ with $\xi \in V$ and $z \in Z$, one has
\begin{align}
\|z\|_X &\le C\|M(x)\|_Y, \label{eq:z_bound} \\
\|M(c(\xi))\|_Y &\le C\|M(x)\|_Y, \label{eq:M_bound} \\
|E(x)-\Gamma(\xi)| &\le C\|M(x)\|_Y^2. \label{eq:E_bound}
\end{align}
\end{lemma}

\begin{proof}
To prove \eqref{eq:z_bound}, define the transverse map $G(\xi,z) := R M(c(\xi)+z)$. By the Lyapunov--Schmidt construction, $R M(c(\xi)) = 0$, so $G(\xi,0) = 0$. By the fundamental theorem of calculus,
$$G(\xi,z) = \left( \int_0^1 R DM(c(\xi)+sz)|_Z \, ds \right) z. $$
At $(\xi,z) = (0,0)$, the bounded linear operator in parentheses is precisely $R DM(0)|_Z = R L|_Z = L|_Z$. Since $L|_Z \colon Z \to \operatorname{Ran} L$ is an isomorphism, continuity implies that for $x$ sufficiently close to $0$, the integral operator remains an isomorphism with a uniformly bounded inverse. Therefore $\|z\|_X\le C_0\|G(\xi,z)\|_Y=C_0\|RM(x)\|_Y\le C_1\|M(x)\|_Y$, which establishes the first estimate.

For \eqref{eq:M_bound}, local Lipschitz continuity of $M$ gives $\|M(x)-M(c(\xi))\|_Y\le C_2\|z\|_X$. Hence, by \eqref{eq:z_bound}, $\|M(c(\xi))\|_Y\le\|M(x)\|_Y+C_2\|z\|_X\le(1+C_1C_2)\|M(x)\|_Y$.

\smallskip
For \eqref{eq:E_bound}, we expand the energy difference using the fundamental theorem of calculus:
$$E(x) - \Gamma(\xi) = E(c(\xi)+z) - E(c(\xi)) = \int_0^1 \langle M(c(\xi)+sz), z \rangle_{Y,X} \, ds. $$
We separate the first-order term at $s=0$:
$$E(x) - \Gamma(\xi) = \langle M(c(\xi)), z \rangle_{Y,X} + \int_0^1 \langle M(c(\xi)+sz) - M(c(\xi)), z \rangle_{Y,X} \, ds.  $$
Because $RM(c(\xi))=0$, we have $M(c(\xi))\in Y_K=Z^\circ$; since $z\in Z$, the first-order term $\langle M(c(\xi)),z\rangle_{Y,X}$ vanishes.
Let $C_{\mathrm{pair}}>0$ be a continuity constant for the pairing, so that
$|\langle y,x\rangle_{Y,X}|\le C_{\mathrm{pair}}\|y\|_Y\|x\|_X$.
The remaining integral is bounded using the local Lipschitz property of $M$:
$$|E(x) - \Gamma(\xi)| \le C_{\mathrm{pair}}\int_0^1 \|M(c(\xi)+sz) - M(c(\xi))\|_Y \|z\|_X \, ds
\le \frac{C_{\mathrm{pair}}L_{\mathrm{lip}}}{2} \|z\|_X^2.  $$
Applying \eqref{eq:z_bound} yields $|E(x) - \Gamma(\xi)| \le C_3 \|M(x)\|_Y^2$. So, $C := \max(C_1, 1+C_1 C_2, C_3)$ ends the proof.
\end{proof}

\subsection{Exact reduction and the Morse--Bott dichotomy}

Define the finite-dimensional \emph{reduced \L{}ojasiewicz exponent} by
$$\theta_R(\Gamma,0) := \inf \left\{ \theta \in [1/2,1) : \|D\Gamma(\xi)\|_{K^*} \ge C|\Gamma(\xi)|^{\theta} \,  \text{ for all $\xi$ near } 0 \right\}. $$

Feehan's invariance under a nondegenerate quadratic summand and Banach splitting theorem already identify the full and residual exponents in his setting~\cite[Lemma~3.3 and Theorem~5]{Feehan2020}. The following theorem gives a direct proof for the present asymmetric pair $X$--$Y$ and, more strongly, identifies the sets of admissible exponents.

\begin{theorem} \label{thm:exact_reduction}
Under the hypotheses above, we have
$$\theta_{\mathrm{LS}}(E,0) = \theta_R(\Gamma,0).$$
\end{theorem}

\begin{proof}
We prove that the sets of admissible exponents coincide.

{\it Step 1: An exponent for $\Gamma$ is an exponent for $E$: }
Suppose $\theta \in [1/2,1)$ is admissible for $\Gamma$. Thus, for sufficiently small $\xi$, $\|D\Gamma(\xi)\|_{K^*} \ge C_0|\Gamma(\xi)|^{\theta}$. By the norm equivalence from Lemma~\ref{lem:gradient_identity}, $\|M(c(\xi))\|_Y\ge C_1\|D\Gamma(\xi)\|_{K^*}$, and therefore $|\Gamma(\xi)|^\theta\le C_2\|M(c(\xi))\|_Y$. Using \eqref{eq:M_bound} from Lemma \ref{lem:comparison}, we obtain
$|\Gamma(\xi)|^{\theta} \le C_3\|M(x)\|_Y$.
On the other hand, \eqref{eq:E_bound} gives $|E(x)|\le|\Gamma(\xi)|+C\|M(x)\|_Y^2$. Since $0<\theta<1$, subadditivity of the fractional power yields $|E(x)|^\theta\le|\Gamma(\xi)|^\theta+C^\theta\|M(x)\|_Y^{2\theta}$.
Substituting the gradient bound for $|\Gamma(\xi)|^{\theta}$ yields $|E(x)|^{\theta} \le C_3\|M(x)\|_Y + C^{\theta}\|M(x)\|_Y^{2\theta}$.

\smallskip
Because $2\theta\ge1$, after shrinking so that $\|M(x)\|_Y\le1$ we have $\|M(x)\|_Y^{2\theta}\le\|M(x)\|_Y$. Hence $|E(x)|^\theta\le(C_3+C^\theta)\|M(x)\|_Y$, so $\theta$ is admissible for $E$.

\smallskip
{\it Step 2: An exponent for $E$ is an exponent for $\Gamma$: }
Conversely, suppose $\|M(x)\|_Y \ge C_0|E(x)|^{\theta}$ for all $x$ sufficiently close to $0$. Restrict this inequality to $x=c(\xi)$. Since $E(c(\xi))=\Gamma(\xi)$, we have $\|M(c(\xi))\|_Y\ge C_0|\Gamma(\xi)|^\theta$. By Lemma~\ref{lem:gradient_identity} and the fact that $\rho$ is an isomorphism, $\|D\Gamma(\xi)\|_{K^*}\ge C_1\|M(c(\xi))\|_Y\ge C_2|\Gamma(\xi)|^\theta$.
Thus $\theta$ is admissible for $\Gamma$. We conclude: since the two sets of admissible exponents are therefore identical, so their infima coincide: $\theta_{\mathrm{LS}}(E,0) = \theta_R(\Gamma,0)$.
\end{proof}

The order estimate below is an elementary finite-dimensional consequence of the first nonzero homogeneous term of the reduced germ. We include the short argument because it will be used repeatedly.

\begin{theorem} \label{thm:order_bound} 
Assume the hypotheses of Theorem~\ref{thm:exact_reduction}, and suppose that $\Gamma \not\equiv 0$. Let $m := \operatorname{ord}_0 \Gamma$ denote the order of vanishing of the reduced functional at $0$. Then $m \ge 3$, and 
$$ \theta_{\mathrm{LS}}(E,0) = \theta_R(\Gamma,0) \ge \frac{m-1}{m}. $$
\end{theorem}  

\begin{proof} 
We first note that the quadratic part of the reduced energy vanishes. For $h,k\in K$, differentiating
$D\Gamma(\xi)[h]=\langle M(c(\xi)),h\rangle_{Y,X}$ at $\xi=0$ in the direction $k$ gives $D^2\Gamma(0)[h,k]=\langle Lk,h\rangle_{Y,X}=0$, because $k\in K=\ker L$. Thus $D\Gamma(0)=D^2\Gamma(0)=0$. Since $\Gamma\not\equiv0$,
the order of vanishing $m=\operatorname{ord}_0\Gamma$ is therefore a finite integer satisfying $m\ge3$.
Because $K$ is finite-dimensional and $\Gamma$ is real analytic, its Taylor expansion at $0$ has the form 
$$\Gamma(\xi) = \Gamma_m(\xi)+\Gamma_{m+1}(\xi)+\cdots, $$
where $\Gamma_j$ is homogeneous of degree $j$ and $\Gamma_m\not\equiv0$.  

\smallskip
Choose $a\in K$ such that $\Gamma_m(a)\neq0$. Since $\Gamma_m$ is homogeneous of degree $m$, Euler's identity gives $D\Gamma_m(a)[a] = m\Gamma_m(a)\neq0$. Consequently, $D\Gamma_m(a)\neq0$ as an element of $K^*$.  
Consider now the real analytic arc $\xi(t)=ta$. By homogeneity of the Taylor coefficients, 
$$\Gamma(ta) = t^m\Gamma_m(a)+O(|t|^{m+1}), $$
and, in $K^*$, we obtain $ D\Gamma(ta) = t^{m-1}D\Gamma_m(a)+O(|t|^m)$. 
Since $\Gamma_m(a)\neq0$ and $D\Gamma_m(a)\neq0$, there exist constants $c_1,c_2>0$ and $t_0>0$ such that, for $0<|t|<t_0$, 
$$ |\Gamma(ta)|\ge c_1|t|^m \quad \text{and} \quad \|D\Gamma(ta)\|_{K^*}\le c_2|t|^{m-1}. $$  

\smallskip
Now let $\theta\in[1/2,1)$ be any admissible exponent for the reduced functional. Thus there exists $C>0$ such that 
$ \|D\Gamma(\xi)\|_{K^*} \ge C|\Gamma(\xi)|^\theta $
for all sufficiently small $\xi$. Applying this inequality along the arc $\xi=ta$, we obtain 
$$ c_2|t|^{m-1} \ge C|\Gamma(ta)|^\theta \ge Cc_1^\theta |t|^{m\theta}  $$
for all sufficiently small $t\neq0$.  

\smallskip
If $m\theta < m-1$, then division by $|t|^{m\theta}$ gives $c_2|t|^{m-1-m\theta} \ge Cc_1^\theta$, while the left-hand side tends to $0$ as $t\to0$, a contradiction. Hence every admissible exponent satisfies $m\theta\ge m-1$, or equivalently, $\theta\ge\frac{m-1}{m}$. Taking the infimum over all admissible exponents yields 
$\theta_R(\Gamma,0)\ge\frac{m-1}{m}$.  
Finally, Theorem~\ref{thm:exact_reduction} identifies the admissible exponents for the full functional $E$ with those of the reduced functional $\Gamma$. Therefore 
$$\theta_{\mathrm{LS}}(E,0) = \theta_R(\Gamma,0) \ge \frac{m-1}{m}. $$
This proves the theorem. 
\end{proof}

The equivalence between vanishing of the residual germ and the Morse--Bott property is closely related to the Banach-space splitting and Morse--Bott lemmas; see, for example, Feehan~\cite[Theorems~5 and~2.14]{Feehan2020}. We use the standard local definition: $E$ is Morse--Bott at $0$ if $\operatorname{Crit}E$ is a $C^2$ submanifold near $0$ and $T_0\operatorname{Crit}E=\ker DM(0)=K$. We give a direct argument in the present asymmetric notation.

\begin{proposition} \label{prop:morse_bott_reduction}
Under the hypotheses of Theorem~\ref{thm:exact_reduction}, the functional $E$ is Morse--Bott at $0$ if and only if the reduced energy $\Gamma$ is identically zero in a neighborhood of $0 \in K$. 
\end{proposition}

\begin{proof}
Recall that $K = \ker L$ and $X = K \oplus Z$, and that the Lyapunov--Schmidt parametrization is
$c(\xi) = \xi + \psi(\xi)$ for  $\xi \in V \subseteq K$, with $R M(c(\xi)) = 0$. 
Hence $M(c(\xi)) \in Y_K$. By the gradient identity of Lemma~\ref{lem:gradient_identity},
$D\Gamma(\xi) = \rho(M(c(\xi))) \in K^*$. 
Since $\rho:Y_K\to K^*$ is an isomorphism, $D\Gamma(\xi)=0$ if and only if $M(c(\xi))=0$.

\medskip
Assume first that $\Gamma \equiv 0$ near $0$. Then $D\Gamma(\xi) = 0$ for every sufficiently small $\xi \in K$. By this equivalence,
we know $M(c(\xi)) = 0$. Thus $c(V) \subseteq \operatorname{Crit}E$, where $\operatorname{Crit}E = \{x \in U : M(x) = 0\}$.
Conversely, let $x \in \operatorname{Crit}E$ be sufficiently close to $0$. Since $X=K\oplus Z$, we may write uniquely $x=\xi+z$ with $\xi\in K$ and $z\in Z$.
Since $M(x) = 0$, we have in particular $R M(\xi + z) = 0$. 
By the uniqueness part of the implicit function theorem defining $\psi$, the unique sufficiently small $z \in Z$ satisfying this equation is
$z = \psi(\xi)$. Therefore $x = c(\xi)$. Consequently, $\operatorname{Crit}E = c(V)$ near $0$.

\smallskip
Since $c$ has the form $c(\xi) = \xi + \psi(\xi),  \psi(0) = 0$ and $D\psi(0) = 0$, 
it is a real analytic embedding of $V$ as a finite-dimensional submanifold of $X$, since the projection $P:X=K\oplus Z\to K$ satisfies $P\circ c=\operatorname{id}_V$. Moreover, $T_0\operatorname{Crit}E=Dc(0)(K)=K$ because $D\psi(0)=0$.
So, the critical set is a finite-dimensional real analytic submanifold near $0$ and
$T_0 \operatorname{Crit}E = \ker L$. Thus $E$ is Morse--Bott at $0$.

\medskip
Conversely, assume that $E$ is Morse--Bott at $0$. Let $\mathcal{C} := \operatorname{Crit}E$. 
By the Morse--Bott assumption, after shrinking the neighborhood of $0$, $\mathcal{C}$ is a finite-dimensional $C^2$ submanifold of $X$ satisfying
$T_0 \mathcal{C} = K$. In particular, $\dim \mathcal{C} = \dim K$.

\medskip
Let $P : X = K \oplus Z \longrightarrow K$
be the continuous projection associated with the chosen splitting. Consider its restriction
$P|_{\mathcal{C}} : \mathcal{C} \longrightarrow K$. At the origin, $D(P|_{\mathcal C})_0=P|_{T_0\mathcal C}=P|_K=\operatorname{id}_K$.
Since $\mathcal{C}$ and $K$ have the same finite dimension, the finite-dimensional inverse function theorem implies that, after shrinking $\mathcal{C}$, the map
$P|_{\mathcal{C}} : \mathcal{C} \longrightarrow V'$
is a local diffeomorphism onto a neighborhood $V' \subset K$ of $0$.

\medskip
Thus, for every $\xi \in V'$, there is a unique critical point $x_\xi \in \mathcal{C}$ such that
$P x_\xi = \xi$. Writing $x_\xi=\xi+z_\xi$ with $z_\xi\in Z$ and using $M(x_\xi)=0$, we obtain $RM(\xi+z_\xi)=0$.
By the uniqueness in the implicit function theorem defining the Lyapunov--Schmidt correction,
$z_\xi = \psi(\xi)$. Hence $x_\xi = c(\xi)$. Since $x_\xi$ is critical,
$M(c(\xi)) = 0$ for every $\xi \in V'$.
Applying the gradient identity, we have $D\Gamma(\xi) = \rho(M(c(\xi))) = 0$
on $V'$. Therefore $\Gamma$ is constant on the connected neighborhood $V'$. Since
$\Gamma(0) = E(c(0)) = E(0) = 0$, we conclude that
$\Gamma \equiv 0$ on $V'$. The proof is complete.
\end{proof}

%The previous result implies a Morse--Bott dichotomy as stated next.

Feehan~\cite[Theorems~1 and~2]{Feehan2020} proves converse results showing, under his analytic Banach hypotheses, that a \L{}ojasiewicz exponent $1/2$ forces the Morse--Bott property. In the present reduction, the additional gap below is an elementary consequence of the vanishing quadratic residual and Theorem~\ref{thm:order_bound}.

\begin{corollary} \label{cor:order_bound}
Assume the hypotheses of Theorem~\ref{thm:exact_reduction}. Exactly one of the following alternatives holds:

\smallskip
\noindent\textnormal{(i)} \quad $\Gamma\equiv0$ near $0$. In this case $E$ is Morse--Bott at $0$ and
$\theta_{\mathrm{LS}}(E,0)=\theta_R(\Gamma,0)=1/2$.

\smallskip
\noindent\textnormal{(ii)} \quad $\Gamma\not\equiv0$. If $m=\operatorname{ord}_0\Gamma$, then
$m\ge3$ and
$$
\theta_{\mathrm{LS}}(E,0)=\theta_R(\Gamma,0)\ge\frac{m-1}{m}\ge\frac23.
$$
In particular, $\theta_{\mathrm{LS}}(E,0)\notin(1/2,2/3)$.
\end{corollary}

\begin{proof}
If $\Gamma\equiv0$, Proposition~\ref{prop:morse_bott_reduction} shows that $E$ is Morse--Bott. Moreover,
Lemma~\ref{lem:comparison} gives $|E(x)|\le C\|M(x)\|_Y^2$, so $1/2$ is admissible and hence
$\theta_{\mathrm{LS}}(E,0)=1/2$. Conversely, if $E$ is Morse--Bott, Proposition~\ref{prop:morse_bott_reduction}
gives $\Gamma\equiv0$. If $\Gamma\not\equiv0$, Theorem~\ref{thm:order_bound} gives
$m\ge3$ and $\theta_{\mathrm{LS}}(E,0)\ge(m-1)/m\ge2/3$.
\end{proof}

Section \ref{sec:framework} has therefore reduced the optimal Łojasiewicz--Simon exponent exactly to the finite-dimensional germ $\Gamma$, and the first nonzero homogeneous term of $\Gamma$ already imposes a universal lower bound. At this stage, however, $\Gamma$ is still constructed using a choice of compatible splitting. Before attaching algebraic or valuative invariants to the reduced germ, we must show that its singularity type is independent of those auxiliary choices in a sufficiently rigid sense. This is the purpose of the next section.

\section{Intrinsic reduction and the Simon pair}\label{sec:intrinsic}

The splitting lemma and the finite-dimensional residual germ attached to a degenerate critical point are classical. In particular, Feehan's Banach splitting theorem for degenerate critical points~\cite[Theorem~5]{Feehan2020}, together with invariance under a nondegenerate quadratic summand~\cite[Lemma~3.3]{Feehan2020}, shows that the residual germ carries the same \L{}ojasiewicz exponent as the full germ in his setting. In finite-dimensional singularity theory, uniqueness of the residual part up to right equivalence belongs to the splitting-lemma picture; Greuel--Pfister~\cite{GreuelPfister2026} give a recent treatment, including convergent real and complex analytic power series.

There are nevertheless two reasons to revisit uniqueness in our setting. First, the gradient in Section~\ref{sec:framework} is a map $M:X\to Y$ and the Hessian is an operator $L:X\to Y$ defined relative to a pairing, rather than a self-adjoint endomorphism of a Hilbert space. The transverse equation therefore depends simultaneously on a complement $Z\subset X$ and on a compatible complement $Y_K\subset Y$. Second, the finite-dimensional space on which every reduction is written is not an abstract residual space: it is the fixed canonical vector space $K=\ker L$. Ordinary uniqueness up to right equivalence would identify the singularity only after an arbitrary linear change of coordinates on $K$. For the leading obstruction used later, we need to know whether the linear part of that change can be taken to be the identity. The theorem below gives exactly this stronger conclusion by deforming one compatible splitting to the other.

\subsection{Independence of the compatible reduction}

\begin{theorem}
\label{thm:right-equivalence}
Assume the hypotheses of Section~\ref{sec:framework}. Let
$X=K\oplus Z_0=K\oplus Z_1$ be two closed complements of $K=\ker L$, and suppose that for $i=0,1$ there are compatible splittings
$Y=\operatorname{Ran}L\oplus Y_{K,i}$ with $Y_{K,i}=Z_i^\circ$ and restriction maps $\rho_i:Y_{K,i}\to K^*$ isomorphisms. If $\Gamma_0,\Gamma_1:(K,0)\to(\mathbb R,0)$ are the corresponding reduced energies, then there is a real analytic local diffeomorphism $\varphi:(K,0)\to(K,0)$ such that
$$
\Gamma_1\circ\varphi=\Gamma_0,\qquad \varphi(0)=0,\qquad D\varphi(0)=\operatorname{id}_K.
$$
\end{theorem}

\begin{proof}
Because $Z_0$ and $Z_1$ complement the same finite-dimensional space $K$, there is a bounded linear map $S:Z_0\to K$ such that $Z_1=\{z+Sz:z\in Z_0\}$. Put $T_tz=z+tSz$ and $Z_t=T_tZ_0$ for $0\le t\le1$; then $X=K\oplus Z_t$ for every $t$.

Let $q:Y\to K^*$ be restriction to $K$ and let $s_0=\rho_0^{-1}:K^*\to Y_{K,0}$. Since $q$ vanishes on $\operatorname{Ran}L=K^\circ$, every element of $Y_{K,1}$ is uniquely of the form $s_0(\lambda)+H\lambda$, where $H:K^*\to\operatorname{Ran}L$ is linear. The condition $Y_{K,1}=Z_1^\circ$ gives, for $z\in Z_0$,
$\langle H\lambda,z\rangle_{Y,X}=-\lambda(Sz)$. Define
$Y_{K,t}=\{s_0(\lambda)+tH\lambda:\lambda\in K^*\}$. Then $Y_{K,t}\subseteq Z_t^\circ$. Conversely, restriction $q:Z_t^\circ\to K^*$ is injective: an element in its kernel annihilates both $K$ and $Z_t$, hence all of $X$, and is zero because the pairing separates points of $Y$. Thus $\dim Z_t^\circ\le\dim K$, while $Y_{K,t}\subseteq Z_t^\circ$ already has dimension $\dim K$. Therefore $Y_{K,t}=Z_t^\circ$. Moreover, $q|_{Y_{K,t}}:Y_{K,t}\to K^*$ is an isomorphism. Given $y\in Y$, choose $y_K\in Y_{K,t}$ with $q(y_K)=q(y)$. Then $q(y-y_K)=0$, so $y-y_K\in K^\circ=\operatorname{Ran}L$ by Lemma~\ref{lem:compatibility-consequences}. Since $q$ is injective on $Y_{K,t}$, also $Y_{K,t}\cap\operatorname{Ran}L=\{0\}$. Hence
$$
Y=\operatorname{Ran}L\oplus Y_{K,t}.
$$

Let $R_t$ be the projection onto $\operatorname{Ran}L$ along $Y_{K,t}$. Relative to $Y=\operatorname{Ran}L\oplus Y_{K,0}$ one has $R_t=R_0-tHq$, so $R_t$ depends analytically on $t$. Consider
$F(t,\xi,z)=R_tM(\xi+T_tz)$ for $\xi\in K$ and $z\in Z_0$. Since $LSz=0$,
$D_zF(t,0,0)z=R_tLT_tz=Lz$, the fixed isomorphism $L|_{Z_0}:Z_0\to\operatorname{Ran}L$. The analytic implicit function theorem with parameter gives such a jointly analytic solution on a neighborhood of each $t_0\in[0,1]$. By uniqueness these local families agree on overlaps; compactness of $[0,1]$ and a finite subcover allow the $K$-neighborhood to be chosen uniformly in $t$. Thus there is a common neighborhood of $0\in K$ and a jointly analytic family $\psi_t(\xi)\in Z_0$ satisfying
$$
\psi_t(0)=0,
\qquad
R_tM(\xi+T_t\psi_t(\xi))=0.
$$

Set $c_t(\xi)=\xi+T_t\psi_t(\xi)$ and $\Gamma_t=E\circ c_t$. At $t=0,1$ these are the two given reductions. Differentiating the reduction equation at $\xi=0$ gives $D_\xi\psi_t(0)=0$, hence $c_t(0)=0$ and $D_\xi c_t(0)=\operatorname{id}_K$. At $\xi=0$ one has $D_\xi c_t(0)(K)=K$, so $X=D_\xi c_t(0)(K)\oplus Z_t$ for every $t\in[0,1]$. Complementarity is an open condition, and compactness of $[0,1]$ therefore permits the $K$-neighborhood to be shrunk once so that
$X=D_\xi c_t(\xi)(K)\oplus Z_t$ for all $t\in[0,1]$ and all $\xi$ in that common neighborhood. Hence there are unique jointly analytic maps $V_t(\xi)\in K$ and $w_t(\xi)\in Z_t$ such that
$\partial_tc_t(\xi)=D_\xi c_t(\xi)V_t(\xi)+w_t(\xi)$. Since $M(c_t(\xi))\in Y_{K,t}=Z_t^\circ$, the $w_t$ term is annihilated and
$\partial_t\Gamma_t(\xi)=D\Gamma_t(\xi)[V_t(\xi)]$.

Let $\varphi_t$ solve the finite-dimensional time-dependent ODE
$\dot\varphi_t=-V_t\circ\varphi_t$, $\varphi_0=\operatorname{id}$. After shrinking the initial neighborhood, the flow exists for $0\le t\le1$ and
$\frac{d}{dt}\Gamma_t(\varphi_t(\xi))=0$. Hence $\Gamma_t\circ\varphi_t=\Gamma_0$ and, in particular, $\Gamma_1\circ\varphi_1=\Gamma_0$.

Finally, $c_t(0)=0$ implies $V_t(0)=0$. Differentiate
$$
\partial_tc_t(\xi)=D_\xi c_t(\xi)V_t(\xi)+w_t(\xi)
$$
at $\xi=0$. Because $D_\xi c_t(0)=\operatorname{id}_K$ is independent of $t$, the derivative of the left-hand side is zero. The resulting identity is the sum of a vector in $K$ and a vector in $Z_t$; since $X=K\oplus Z_t$, both components vanish, and in particular $DV_t(0)=0$. The variational equation for the flow therefore gives $D\varphi_t(0)=\operatorname{id}_K$. Taking $\varphi=\varphi_1$ completes the proof.
\end{proof}

\begin{definition}
The \emph{reduced Simon singularity} of $(E,0)$ is the analytic right-equivalence class
$$
\mathfrak S(E,0):=[\Gamma]_{\mathcal R}
$$
of any compatible Lyapunov--Schmidt reduced energy.
\end{definition}

\begin{remark}[Intrinsic consequences]\label{rem:intrinsic-consequences}
Because the right-equivalence in Theorem~\ref{thm:right-equivalence} is tangent to the identity, the order $\operatorname{ord}_0\Gamma$ and the first nonzero homogeneous term $\Gamma_m$ are literally independent of the compatible reduction on the fixed space $K=\ker L$. Moreover, the chain rule gives $J(\Gamma_0)=\varphi^*J(\Gamma_1)$ and $(\Gamma_0)=\varphi^*(\Gamma_1)$, so the analytic isomorphism class of the Simon pair introduced below is intrinsic. Newton polyhedra themselves depend on coordinates and will be used only as computational presentations of this intrinsic pair.
\end{remark}

Theorem \ref{thm:right-equivalence} is what allows us to pass from Lyapunov--Schmidt reduction to singularity theory without retaining the auxiliary choices made in the Banach-space splitting. We may henceforth regard $\Gamma$ not as a particular reduced representative, but as an intrinsic finite-dimensional singularity attached to $(E,0)$. We now identify the algebraic datum that records exactly the relation between its energy and its gradient.

\subsection{The intrinsic Simon pair}

Fix linear coordinates $K\cong\mathbb R^r$ and let $A_{\mathbb R}=\mathbb R\{x_1,\ldots,x_r\}$. For an ideal $I=(g_1,\ldots,g_s)$ and a real analytic arc $\gamma:(\mathbb R,0)\to(\mathbb R^r,0)$, write $\nu_\gamma(I)=\min_j\operatorname{ord}_t(g_j\circ\gamma)$. We use the real integral closure characterized by the real curve criterion,
$h\in\overline I^{\,\mathbb R}$ if and only if $\operatorname{ord}_t(h\circ\gamma)\ge\nu_\gamma(I)$ for every real analytic arc $\gamma$; equivalently, $|h|\le C\max_j|g_j|$ near the origin; see Gaffney--Ara\'ujo dos Santos~\cite[Definition~2.1 and Proposition~2.2]{GaffneyAraujo2010}. For $0\ne h\in A_{\mathbb R}$ with $h(0)=0$, set
$$
\mathcal L^{\mathbb R}_{(h)}(I)
:=\inf\left\{\frac qp\in\mathbb Q_{>0}:h^q\in\overline{I^p}^{\,\mathbb R}\right\}.
$$

For a nonzero reduced energy $\Gamma$, define its \emph{Simon pair} by
$$
\mathcal P(\Gamma):=(J(\Gamma),(\Gamma)),
\qquad
J(\Gamma)=\left(\frac{\partial\Gamma}{\partial x_1},\ldots,
\frac{\partial\Gamma}{\partial x_r}\right).
$$
The name is new notation for the classical relative-ideal problem naturally selected by Simon's gradient inequality: $J(\Gamma)$ measures the gradient, while $(\Gamma)$ measures the energy. By Theorem~\ref{thm:right-equivalence}, its analytic isomorphism class is intrinsic to $(E,0)$.

The following is the central bridge from the Banach functional to the classical valuative theory. The arc/inequality equivalence for real integral closure is due to Gaffney--Ara\'ujo dos Santos~\cite{GaffneyAraujo2010}; the real-analytic blowup and rationality mechanism is developed by Bochnak--Risler and in Risler's appendix to Lejeune-Jalabert--Teissier~\cite{BochnakRisler1975,LejeuneTeissier2008}. The finite-max viewpoint is also a special case of the modern relative-threshold framework in \cite{Ha2026}.

\begin{theorem}
\label{thm:Simon-pair-formula}
Assume the hypotheses of Section~\ref{sec:framework} and suppose $\Gamma\not\equiv0$. Let $\pi:\widetilde K\to(K,0)$ be a common real analytic principalization of $J(\Gamma)$ and $(\Gamma)$, whose existence follows from real-analytic principalization in characteristic zero, for example from Bierstone--Milman~\cite{BierstoneMilman1997}, and let $F_1,\ldots,F_\ell$ be the irreducible components of the resulting normal-crossing divisor whose images contain $0$. Then
$$
\theta_{\mathrm{LS}}(E,0)
=\mathcal L^{\mathbb R}_{(\Gamma)}(J(\Gamma))
=\sup_\gamma\frac{\nu_\gamma(J(\Gamma))}{\nu_\gamma(\Gamma)}
=\max_{\operatorname{ord}_{F_i}(\Gamma)>0}
\frac{\operatorname{ord}_{F_i}(J(\Gamma))}{\operatorname{ord}_{F_i}(\Gamma)},
$$
where the supremum is over real analytic arcs with $\Gamma\circ\gamma\not\equiv0$. In particular, the optimal Banach \L{}ojasiewicz--Simon exponent is rational and is attained by a divisorial order on the intrinsic reduced singularity.
\end{theorem}

\begin{proof}
By Theorem~\ref{thm:exact_reduction}, it is enough to compute the finite-dimensional gradient exponent of $\Gamma$. After choosing coordinates on $K$, all norms on $K^*$ are equivalent, so $\|D\Gamma\|_{K^*}$ is comparable with $\max_j|\partial\Gamma/\partial x_j|$. Hence an exponent $q/p$ is admissible exactly when, after shrinking the neighborhood,
$$
|\Gamma|^q\le C\max_j\left|\frac{\partial\Gamma}{\partial x_j}\right|^p.
$$
By the real integral-closure criterion of Gaffney--Ara\'ujo dos Santos, this is equivalent to
$$
\Gamma^q\in\overline{J(\Gamma)^p}^{\,\mathbb R},
$$
and, by the real curve criterion, to
$$
q\,\nu_\gamma(\Gamma)\ge p\,\nu_\gamma(J(\Gamma))
$$
for every real analytic arc $\gamma$ with $\Gamma\circ\gamma\not\equiv0$. The admissible exponents form an upper interval and the positive rationals are dense, so taking the infimum over rational $q/p$ gives the same infimum as in the definition of $\theta_R$. Therefore
$$
\theta_R(\Gamma,0)
=\mathcal L^{\mathbb R}_{(\Gamma)}(J(\Gamma))
=\sup_\gamma\frac{\nu_\gamma(J(\Gamma))}{\nu_\gamma(\Gamma)}.
$$

It remains to justify the finite divisorial maximum. On a common real analytic principalization $\pi$, around each point of $\pi^{-1}(0)$ there are local coordinates $y_1,\ldots,y_s$ and analytic units $u,v$ such that
$$
\pi^*(\Gamma)=u\,y_1^{a_1}\cdots y_s^{a_s},
\qquad
\pi^{-1}J(\Gamma)\cdot\mathcal O_{\widetilde K}
=(v\,y_1^{b_1}\cdots y_s^{b_s}).
$$
Thus the pulled-back monomial quotient is locally bounded if and only if $qa_j\ge pb_j$ for every divisor component through the chart. To compare the principal generator with the original Jacobian generators, write $J(\Gamma)=(g_1,\ldots,g_N)$. On the chart there are analytic functions $a_j$ such that $g_j\circ\pi=a_jv\,y_1^{b_1}\cdots y_s^{b_s}$. Conversely, because these pullbacks generate the principal ideal $(v\,y_1^{b_1}\cdots y_s^{b_s})$, there are analytic functions $c_j$ such that
$v\,y_1^{b_1}\cdots y_s^{b_s}=\sum_j c_j(g_j\circ\pi)$. Hence, on every relatively compact subchart meeting $\pi^{-1}(0)$,
$$
c\,|v\,y_1^{b_1}\cdots y_s^{b_s}|
\le \max_j|g_j\circ\pi|
\le C\,|v\,y_1^{b_1}\cdots y_s^{b_s}|
$$
for suitable positive constants $c,C$. If the divisor inequalities hold for all components meeting $\pi^{-1}(0)$, the quotient
$$
\frac{|\pi^*\Gamma|^q}{|v\,y_1^{b_1}\cdots y_s^{b_s}|^p}
$$
is locally bounded. Properness of $\pi$ and compactness of the fiber $\pi^{-1}(0)$ give finitely many such charts and uniform constants after shrinking the base. The two-sided comparison above then gives the required inequality in terms of $\max_j|g_j|$ downstairs. Conversely, any analytic inequality downstairs pulls back and forces $qa_j\ge pb_j$ along each component. Therefore
$$
\mathcal L^{\mathbb R}_{(\Gamma)}(J(\Gamma))
=\max_{\operatorname{ord}_{F_i}(\Gamma)>0}
\frac{\operatorname{ord}_{F_i}(J(\Gamma))}{\operatorname{ord}_{F_i}(\Gamma)}.
$$
This is the real-analytic blowup mechanism of Bochnak--Risler and Risler~\cite{BochnakRisler1975,LejeuneTeissier2008}. The maximum is finite and rational, completing the proof.
\end{proof}

\begin{corollary}\label{cor:order-sharpness}
Let $m=\operatorname{ord}_0\Gamma$. Under the hypotheses of Theorem~\ref{thm:Simon-pair-formula},
$$
\theta_{\mathrm{LS}}(E,0)=\frac{m-1}{m}
\quad\Longleftrightarrow\quad
\Gamma^{m-1}\in\overline{J(\Gamma)^m}^{\,\mathbb R}.
$$
\end{corollary}

\begin{proof}
The containment says exactly that $(m-1)/m$ is admissible, while Theorem~\ref{thm:order_bound} gives the reverse inequality. Conversely, if the optimal exponent equals $(m-1)/m$, Theorem~\ref{thm:Simon-pair-formula} shows that the optimum is attained and hence admissible, giving the stated containment.
\end{proof}

Theorem \ref{thm:Simon-pair-formula} gives an intrinsic description of the exponent, but not yet a directly computable one: it expresses the answer through real analytic arcs or through a principalization of the Simon pair. We next relate this description to more explicit complex-algebraic and Newton-polyhedral calculations. The essential issue in passing from the complex problem back to the real \L{}ojasiewicz--Simon exponent is whether a valuation computing the complex exponent is visible on the real locus.

\subsection{Complex and Newton computation}

Let $\Gamma_{\mathbb C}\in\mathbb C\{z_1,\ldots,z_r\}$ be the complexification of $\Gamma$ and define the complex relative exponent
$\theta_{\mathbb C}(\Gamma)=\mathcal L^{\mathbb C}_{(\Gamma_{\mathbb C})}(J(\Gamma_{\mathbb C}))$ using ordinary integral closure. Restriction of a complex gradient inequality to the real locus immediately gives
$\theta_{\mathrm{LS}}(E,0)\le\theta_{\mathbb C}(\Gamma)$. By the classical Rees valuative criterion,
$$
\theta_{\mathbb C}(\Gamma)
=\max_v\frac{v(J(\Gamma_{\mathbb C}))}{v(\Gamma_{\mathbb C})},
$$
where $v$ runs through the Rees valuations of $J(\Gamma_{\mathbb C})$; see \cite{HunekeSwanson2006,Ha2026}. Thus the only extra issue in passing from the complex algebraic problem to the actual real \L{}ojasiewicz--Simon exponent is whether the extremal complex valuation is visible on the real locus. One useful sufficient mechanism is completely elementary: if the complex maximum is computed by finitely many monomial valuations in real coordinates, then the real and complex exponents coincide. Indeed, for any positive integral weight $u$, a point $c\in(\mathbb R^*)^r$ can be chosen outside the zero sets of the finitely many $u$-initial forms of $\Gamma$ and the chosen generators of $J(\Gamma)$. The monomial arc $t\mapsto(c_1t^{u_1},\ldots,c_rt^{u_r})$ then realizes the same weighted orders on the real locus. We will use precisely this observation in the Newton case below.

Bivi\`a-Ausina~\cite{Bivia2003} develops Newton-polyhedral estimates for relative \L{}ojasiewicz exponents, while \cite{Bivia2005} studies Newton-nondegenerate ideals, their use in computing \L{}ojasiewicz exponents, and in particular Newton nondegeneracy of Jacobian ideals. Applying the standard monomial/Rees-valuation mechanism to the Simon pair gives the following convenient computational consequence.

\begin{corollary}
\label{cor:Newton-Simon}
Assume $\Gamma_{\mathbb C}$ has an isolated critical point at the origin and $J(\Gamma_{\mathbb C})$ is Newton nondegenerate. Let $\mathcal N(J(\Gamma_{\mathbb C}))$ be the primitive inward normals to the compact facets of the Newton polyhedron of $J(\Gamma_{\mathbb C})$. Then
$$
\theta_{\mathrm{LS}}(E,0)
=\theta_{\mathbb C}(\Gamma)
=\max_{u\in\mathcal N(J(\Gamma_{\mathbb C}))}
\frac{\nu_u(J(\Gamma))}{\nu_u(\Gamma)}.
$$
Moreover, every maximizing monomial valuation in this formula is realized simultaneously on $J(\Gamma)$ and $\Gamma$ by a real analytic monomial arc.
\end{corollary}

\begin{proof}
The isolated-critical-point hypothesis makes $J(\Gamma_{\mathbb C})$ $\mathfrak m$-primary. By Bivi\`a-Ausina's Newton-nondegeneracy criterion~\cite[Theorem~3.4]{Bivia2005}, its integral closure equals that of its Newton monomial model, whose Rees valuations are the monomial valuations associated to the primitive inward normals of the compact facets~\cite{HunekeSwanson2006}. Complexification does not change the supports of $\Gamma$ or its partial derivatives, so the corresponding weighted orders agree over $\mathbb R$ and $\mathbb C$. For a fixed positive integral weight $u$, choose $c\in(\mathbb R^*)^r$ outside the zero sets of the finitely many $u$-initial forms of $\Gamma$ and a set of generators of $J(\Gamma)$. Then the real arc $t\mapsto(c_1t^{u_1},\ldots,c_rt^{u_r})$ has exactly the prescribed weighted orders. Hence the real arc supremum in Theorem~\ref{thm:Simon-pair-formula} is at least every facet ratio, while restriction from the complex inequality gives the opposite inequality.
\end{proof}

For computations one may read the two weighted orders directly from the support. If $\Gamma(x)=\sum_\alpha c_\alpha x^\alpha$ and $u\in\mathbb Z_{>0}^r$, then
$$
\nu_u(\Gamma)=\min_{c_\alpha\ne0}\langle u,\alpha\rangle,
\qquad
\nu_u(J(\Gamma))=
\min_i\min_{\substack{c_\alpha\ne0\\ \alpha_i>0}}
(\langle u,\alpha\rangle-u_i),
$$
with the inner minimum interpreted as $+\infty$ when the corresponding partial derivative vanishes identically. These are elementary support identities; no separate Newton theorem is involved.

\begin{example}
\label{ex:x3-y100}
For $\Gamma(x,y)=x^3+y^{100}$, one has $m=3$, so Theorem~\ref{thm:order_bound} gives only $\theta_R(\Gamma,0)\ge2/3$. But $J(\Gamma)=(3x^2,100y^{99})$ is monomial. Its compact Newton edge has primitive inward normal $u=(99,2)$, for which $\nu_u(J(\Gamma))=198$ and $\nu_u(\Gamma)=200$. Hence Corollary~\ref{cor:Newton-Simon} gives $\theta_R(\Gamma,0)=99/100$. The same ratio is already visible on the real $y$-axis. Thus higher-order terms may create much slower directions when the leading homogeneous obstruction has a non-isolated critical set.
\end{example}

The preceding example shows in particular why the order of the reduced germ alone need not determine the optimal exponent: singular directions of its leading term may expose much higher-order behavior. To show that the valuative framework can nevertheless lead to closed formulas even for highly singular germs, we now turn to a determinantal family of sum-of-squares singularities. Besides providing a substantial class of explicit Simon-pair computations, this family will supply the algebraic model for the higher-dimensional $SU(2)$ Yang--Mills application.

\subsection{Determinantal Simon pairs}\label{subsec:determinantal}

The preceding valuative description admits a closed-form calculation for a broad family of determinantal sum-of-squares singularities; we use standard determinantal notation as in~\cite{BrunsVetter1988}. Let $\operatorname{Mat}_{m,n}(\mathbb R)$ be the space of real $m\times n$ matrices and put $r=\min\{m,n\}$. For $0\le j\le r$, let $I_j$ denote the ideal generated by the $j\times j$ minors of the generic matrix, with $I_0=(1)$, and define
$$
Q_j(X):=\sum_{|A|=|B|=j}\det(X_{A,B})^2,
\qquad Q_0=1,
$$
where the sum ranges over all $j$-element row and column subsets. We set $Q_j=0$ for $j>r$.

\begin{lemma}\label{lem:det-gradient}
For $1\le k\le r$ one has the polynomial identity
$$
\|\nabla Q_k(X)\|_F^2
=4\sum_{\ell=0}^{k-1}(2\ell+1)Q_{k-1-\ell}(X)Q_{k+\ell}(X),
$$
where $\|\cdot\|_F$ is the Frobenius norm.
\end{lemma}

\begin{proof}
Both sides are invariant under left and right orthogonal transformations and are polynomial in the entries of $X$, so it is enough to verify the identity on the dense set of matrices with distinct positive singular values. Write these singular values as $\sigma_1,\ldots,\sigma_r$ and put $a_i=\sigma_i^2$. By Cauchy--Binet, $Q_j(X)=e_j(a_1,\ldots,a_r)$, where $e_j$ is the $j$-th elementary symmetric polynomial. At a diagonal singular-value matrix, the stabilizing sign changes in the left and right orthogonal groups force all off-diagonal entries of the gradient to vanish (and likewise the entries in any rectangular zero block). Its $i$-th diagonal entry is
$\partial Q_k/\partial\sigma_i=2\sigma_i e_{k-1}(a_1,\ldots,\widehat a_i,\ldots,a_r)$. Hence
$$
\frac14\|\nabla Q_k\|_F^2
=\sum_{i=1}^r a_i e_{k-1}(\widehat a_i)^2.
$$
It remains to prove the symmetric-function identity
$$
\sum_i a_i e_{k-1}(\widehat a_i)^2
=\sum_{\ell=0}^{k-1}(2\ell+1)e_{k-1-\ell}e_{k+\ell}.
$$
Let $E(z)=\prod_i(1+a_i z)=\sum_{j\ge0}e_jz^j$. The left-hand side is the coefficient of $x^{k-1}y^{k-1}$ in
$$
E(x)E(y)\sum_i\frac{a_i}{(1+a_ix)(1+a_iy)}
=\frac{xE'(x)E(y)-yE'(y)E(x)}{x-y}.
$$
Pairing the terms with bidegrees $(p,q)$ and $(q,p)$, $p>q$, gives
$$
(p-q)e_pe_q\frac{x^py^q-x^qy^p}{x-y}.
$$
A diagonal monomial $x^{k-1}y^{k-1}$ occurs in this expression precisely when $p+q=2k-1$, and then occurs once, with coefficient $p-q$. Writing $p=k+\ell$ and $q=k-1-\ell$ gives exactly the displayed sum. This proves the identity.
\end{proof}

\begin{theorem}\label{thm:determinantal-Simon}
Fix $1\le k\le r$. In the real analytic local ring at any matrix $X_0$ one has
$$
\overline{(Q_k)}^{\,\mathbb R}=\overline{I_k^2}^{\,\mathbb R},
\qquad
\overline{J(Q_k)}^{\,\mathbb R}=\overline{I_{k-1}I_k}^{\,\mathbb R}.
$$
If $\operatorname{rank}X_0=s<k$, then the optimal \L{}ojasiewicz gradient exponent of $Q_k$ at $X_0$ is
$$
1-\frac{1}{2(k-s)}.
$$
In particular, at the origin the exponent is $(2k-1)/(2k)$, whereas on the rank-$(k-1)$ stratum it is $1/2$.
\end{theorem}

\begin{proof}
Let $\gamma:(\mathbb R,0)\to(\operatorname{Mat}_{m,n}(\mathbb R),X_0)$ be a real analytic arc. If $I_k\circ\gamma\equiv0$, then $Q_k\circ\gamma\equiv0$ and, because $Q_k$ is a sum of squares of the $k$-minors, $J(Q_k)\circ\gamma\equiv0$ as well. Thus assume that $I_k\circ\gamma\not\equiv0$.

Over the discrete valuation ring $\mathbb R\{t\}$, Smith normal form gives invertible analytic row and column operations; such operations do not change the orders of the determinantal ideals. The nonzero invariant factors have orders
$0\le\alpha_1\le\cdots\le\alpha_\rho$, where $\rho$ is the generic rank of the arc. Put $A_j=\alpha_1+\cdots+\alpha_j$ for $j\le\rho$ and $A_j=+\infty$ for $j>\rho$. The determinantal ideals satisfy $\nu_\gamma(I_j)=A_j$. Since $Q_j$ is a sum of squares of the $j$-minors, there is no cancellation over a real arc and therefore $\nu_\gamma(Q_j)=2A_j$.

Applying Lemma~\ref{lem:det-gradient} and again using positivity of the sum of squares and of all coefficients gives
$$
\nu_\gamma(J(Q_k))
=\min_{0\le\ell\le k-1}\bigl(A_{k-1-\ell}+A_{k+\ell}\bigr),
$$
where a term is $+\infty$ when $k+\ell>\rho$. For a finite term,
$$
A_{k-1-\ell}+A_{k+\ell}-(A_{k-1}+A_k)
=\sum_{j=1}^{\ell}(\alpha_{k+j}-\alpha_{k-j})\ge0,
$$
so the minimum occurs at $\ell=0$. Hence
$$
\nu_\gamma(J(Q_k))=A_{k-1}+A_k
=\nu_\gamma(I_{k-1}I_k).
$$
Together with $\nu_\gamma(Q_k)=2A_k=\nu_\gamma(I_k^2)$, the real arc criterion gives the two integral-closure identities.

Now suppose $\operatorname{rank}X_0=s<k$ and let $h=k-s$. For every arc not contained in $V(I_k)$, exactly the first $s$ invariant-factor orders are zero, so
$$
\frac{\nu_\gamma(J(Q_k))}{\nu_\gamma(Q_k)}
=\frac12+\frac{A_{k-1}}{2A_k}.
$$
Since $A_k=\alpha_{s+1}+\cdots+\alpha_k$ and $\alpha_k$ is at least the average of these $h$ terms, one has
$A_{k-1}=A_k-\alpha_k\le(h-1)A_k/h$. Thus every arc ratio is at most $1-1/(2h)$. Equality is realized by an arc which, after orthogonal row and column changes putting $X_0$ in singular-value normal form, keeps its $s$ nonzero singular values fixed and turns on $h$ additional diagonal entries linearly in $t$. The real arc characterization of the finite-dimensional gradient exponent used in the proof of Theorem~\ref{thm:Simon-pair-formula}, applied after translating $X_0$ to the origin, yields the stated exponent.
\end{proof}

Theorem~\ref{thm:determinantal-Simon} shows in particular that the nominal homogeneous bound is sharp for $Q_k$ at the vertex even though $Q_k$ has a large non-isolated critical set. The exponent then decreases in discrete steps along the rank stratification.

The determinantal calculation illustrates the algebraic side of the theory: the optimal exponent can be read from the way the Simon pair interacts with the rank stratification. Before turning to the PDE applications, we briefly recall the dynamical meaning of the optimal exponent. This also exposes a distinction that is invisible in the algebraic formula itself: divisorial attainment of the optimal ratio does not automatically imply realization of that divisor by an actual gradient trajectory.

\subsection{Dynamical outlook}

The usual consequences of a \L{}ojasiewicz--Simon inequality for gradient-like evolutions remain available in the asymmetric setting. For an equation $\dot u=-\mathcal A(u)M(u)$ with the standard coercivity and boundedness assumptions on $\mathcal A(u):Y\to X$, a trajectory converging to $0$ decays exponentially when $\theta_{\mathrm{LS}}=1/2$; if $\theta=\theta_{\mathrm{LS}}(E,0)>1/2$, then
$$
E(u(t))-E(0)\le C(1+t)^{-1/(2\theta-1)},
\qquad
\|u(t)\|_X\le C(1+t)^{-(1-\theta)/(2\theta-1)}.
$$
See Simon~\cite{Simon1983}, Chill~\cite{Chill2003}, and Chill--Haraux--Jendoubi~\cite{ChillHarauxJendoubi2009}. These are standard upper bounds; the Simon-pair formula raises the separate question of whether an extremal divisor is dynamically realized.

\begin{problem}
\label{prob:valuative-realization}
Let $v$ be a divisorial valuation computing $\theta_{\mathrm{LS}}(E,0)=v(J(\Gamma))/v(\Gamma)$. Find geometric or dynamical conditions guaranteeing a gradient-like trajectory whose asymptotic orders realize $v$. More generally, determine which extremal divisors of the intrinsic Simon pair are dynamically realizable.
\end{problem}

This is stronger than the existence of a slowly converging trajectory. In Example~\ref{ex:x3-y100}, however, the extremal direction is realized: the Euclidean negative gradient flow preserves the $y$-axis, where $y'=-100y^{99}$ and hence $y(t)\asymp t^{-1/98}$, exactly the decay rate predicted by $\theta_R=99/100$.

With the abstract reduction and its algebraic invariants now in place, we turn to genuinely infinite-dimensional applications. The goal is not merely to recover a \L{}ojasiewicz--Simon inequality, but to identify the reduced singularity sharply enough to determine its optimal exponent. Yang--Mills theory provides a particularly clean first test: on flat tori the kernel consists of explicit harmonic modes, and for $SU(2)$ the resulting reduced energy is precisely a determinantal singularity of the type computed in Theorem \ref{thm:determinantal-Simon}.

\section{Application to the Yang--Mills Functional}\label{sec:YM}

\subsection{Analytic setup and compatible splitting}

Let $G$ be a compact Lie group with Lie algebra $\mathfrak g$, equipped with an $\operatorname{Ad}G$-invariant inner product. Let $\mathbb T^2=\mathbb R^2/\mathbb Z^2$ carry the flat metric, let $P=\mathbb T^2\times G$ be the trivial $G$-bundle, and let $\Theta$ denote the trivial flat connection.

\medskip
We parameterize connections near $\Theta$ by $A=\Theta+a$, where $X=W^{1,p}(\mathbb T^2;\Lambda^1\otimes\mathfrak g)$ and $p>2$.
Since $p>2$ and $\dim\mathbb T^2=2$, Sobolev embedding gives $W^{1,p}\hookrightarrow C^0\hookrightarrow L^\infty$. In particular, the quadratic term $[a\wedge a]$ belongs to $L^p\subset L^2$, and the Yang--Mills energy is well-defined on $X$.

Let $Y:=\bigl(W^{1,p'}(\mathbb T^2;\Lambda^1\otimes\mathfrak g)\bigr)^*$, where $p'$ is the conjugate exponent to $p$; we denote this space by $W^{-1,p}$ in the present convention. Since $p'<2<p$, the inclusion $W^{1,p}\hookrightarrow W^{1,p'}$ is continuous. Hence distributional duality gives a continuous pairing $\langle\cdot,\cdot\rangle_{Y,X}:Y\times X\to\mathbb R$ which separates points of $Y$.

The Yang--Mills energy is
$$
\mathcal E(A)=\frac12\|F_A\|_{L^2}^2,
\qquad
F_{\Theta+a}=da+\frac12[a\wedge a].
$$
To remove the infinitesimal gauge degeneracy at $\Theta$, define the gauge-fixed functional
$$
\mathcal E_{\mathrm{gf}}(a)=\frac12\|F_{\Theta+a}\|_{L^2}^2+\frac12\|d^*a\|_{L^2}^2.
$$
For $b\in X$, integration by parts gives $d\mathcal E_{\mathrm{gf}}(a)[b]=\langle d_{\Theta+a}^*F_{\Theta+a}+dd^*a,b\rangle_{Y,X}$. Accordingly, set $\mathcal M(a)=d_{\Theta+a}^*F_{\Theta+a}+dd^*a$. Thus $d\mathcal E_{\mathrm{gf}}(a)[b]=\langle\mathcal M(a),b\rangle_{Y,X}$.

These analyticity properties are standard in the Yang--Mills \L{}ojasiewicz--Simon theory; see, for example, Feehan--Maridakis~\cite[Proposition~3.1]{FeehanMaridakisMemoirs}. We verify them here in the Sobolev scale used below.

\begin{lemma}\label{lem:YM-analytic}
Let $p>2$. The gauge-fixed functional $\mathcal E_{\mathrm{gf}}:X\to\mathbb R$ and the gauge-fixed map $\mathcal M:X\to Y$ are real analytic. In fact, both are continuous polynomial maps.
\end{lemma}

\begin{proof}
The curvature $F_{\Theta+a}=da+\frac12[a\wedge a]$ is a continuous polynomial map from $X$ into $L^p$. Indeed, $W^{1,p}\hookrightarrow L^\infty$, so $[a\wedge a]\in L^p$. The operator $d^*$ maps $L^p$ continuously to $W^{-1,p}$, while the zero-order bracket contraction involving $a$ and an $L^p$ two-form is a continuous bilinear map $X\times L^p\to L^p\hookrightarrow W^{-1,p}$. Hence $a\mapsto d_{\Theta+a}^*F_{\Theta+a}$ is a continuous polynomial map into $Y$. The term $dd^*a$ is linear from $X$ to $Y$, so $\mathcal M$ is a continuous polynomial. The functional $\mathcal E_{\mathrm{gf}}$ is obtained by integrating $L^2$ inner products of polynomial expressions. Since $L^p\hookrightarrow L^2$ on the compact torus, these multilinear forms are continuous. Thus $\mathcal E_{\mathrm{gf}}$ is also a continuous polynomial, of degree at most four.
\end{proof}

The Hodge--Fredholm assertions below are standard elliptic theory; compare the Sobolev Fredholm framework in Feehan--Maridakis~\cite{FeehanMaridakisMemoirs} and Feehan~\cite{Feehan2022}.

\begin{lemma}\label{lem:YM-Hodge}
The following statements hold.
\begin{enumerate}
\item The linearization of $\mathcal M$ at $0$ is the Hodge Laplacian
$$D\mathcal M(0)=\Delta_H:=d^*d+dd^*:X\to Y.$$
\item Its kernel is
$$
K=\ker\Delta_H
=\{\xi\,dx+\eta\,dy:\xi,\eta\in\mathfrak g\}
\cong\mathfrak g\oplus\mathfrak g,
$$
so $\dim K=2\dim\mathfrak g$.
\item There is a topological direct sum $Y=\operatorname{Ran}\Delta_H\oplus K$. Consequently $\Delta_H:X\to Y$ is Fredholm of index zero.
\item The operator $\Delta_H$ is symmetric with respect to the distributional pairing.
\end{enumerate}
\end{lemma}

\begin{proof}
Since $F_{\Theta+a}=da+\frac12[a\wedge a]$, the linear part of $d_{\Theta+a}^*F_{\Theta+a}$ at $a=0$ is $d^*da$, and the gauge-fixing term is linear. Hence $D\mathcal M(0)=d^*d+dd^*=\Delta_H$.

If $u\in X$ and $\Delta_Hu=0$ in $Y$, then $u$ is distributionally harmonic. Elliptic regularity makes $u$ smooth, and therefore $u$ is a harmonic one-form. On the flat torus with trivial coefficient bundle, harmonic one-forms are constant, giving the displayed description of $K$.

Regarding smooth harmonic forms as elements of $Y$ by the $L^2$ pairing, let $P_H:Y\to K$ be harmonic projection. Standard elliptic solvability in the scale $W^{1,p}\to W^{-1,p}$ gives $\ker P_H=\operatorname{Ran}\Delta_H$ and therefore
$$Y=\operatorname{Ran}\Delta_H\oplus K.$$
The range is closed and has codimension $\dim K$, so $\Delta_H$ is Fredholm of index zero. Finally, distributional integration by parts gives
$$
\langle\Delta_Hu,v\rangle_{Y,X}
=\int_{\mathbb T^2}\bigl(\langle du,dv\rangle+\langle d^*u,d^*v\rangle\bigr)
=\langle\Delta_Hv,u\rangle_{Y,X}.
$$
\end{proof}

Choose an $L^2$-orthonormal basis $\{e_j\}_{j=1}^r$ of $K$, where $r=2\dim\mathfrak g$. Define
$$
P_Xu=\sum_{j=1}^r\langle u,e_j\rangle_{L^2}e_j,
\qquad
P_YT=\sum_{j=1}^r\langle T,e_j\rangle_{Y,X}e_j,
$$
and set $Z=\ker P_X$. Exactly as in the $\mathrm{SU}(2)$ case, $P_X$ and $P_Y$ are bounded, $\ker P_Y=\operatorname{Ran}\Delta_H$, and
$$
X=K\oplus Z,
\qquad
Y=\operatorname{Ran}\Delta_H\oplus K.
$$
Moreover $Z^\circ=K$: every $k\in K$ annihilates $Z$ by $L^2$ orthogonality; conversely, if $T\in Z^\circ$, then for $u=P_Xu+z$ one has
$$
\langle T,u\rangle_{Y,X}
=\sum_{j=1}^r\langle u,e_j\rangle_{L^2}\langle T,e_j\rangle_{Y,X},
$$
so $T$ agrees on the dense subspace $X\subset W^{1,p'}$ with the element
$\sum_j\langle T,e_j\rangle e_j\in K$ and hence equals it. Finally, the restriction map $K\to K^*$ is the finite-dimensional Riesz isomorphism. Thus all compatibility hypotheses of Section~\ref{sec:framework} hold.

\subsection{Exact Lyapunov--Schmidt reduction}

For $G=\mathrm{SU}(2)$, Feehan~\cite[Example~A.2 and Theorem~A.4]{FeehanFlat2019} describes the corresponding harmonic Kuranishi model and the singular local flat-connection space. The calculation below is valid for every compact $G$ and gives the exact reduced energy in the present gauge-fixed formulation.

\begin{lemma}\label{lem:reduction}
The Lyapunov--Schmidt correction vanishes identically. Under the identification $K\cong\mathfrak g\oplus\mathfrak g$,
$$
\Gamma_{\mathrm{gf}}(\xi,\eta)=\kappa\|[\xi,\eta]\|^2,
\qquad
\kappa=\frac12\operatorname{Vol}(\mathbb T^2)>0.
$$
\end{lemma}

\begin{proof}
Let $k=\xi\,dx+\eta\,dy\in K$. Since $k$ is constant, $dk=0$ and $d^*k=0$. With the standard graded bracket convention,
$$
[k\wedge k]=2[\xi,\eta]dx\wedge dy,
\qquad
F_{\Theta+k}=[\xi,\eta]dx\wedge dy.
$$
This curvature is constant, so $d^*F_{\Theta+k}=0$. Hence $\mathcal M(k)=d_{\Theta+k}^*F_{\Theta+k}$ is obtained by algebraic contraction of constant forms and is itself a constant $\mathfrak g$-valued one-form. Thus $\mathcal M(k)\in K$.

Let $R=I-P_Y$ be the projection onto $\operatorname{Ran}\Delta_H$ along $K$. Then $R\mathcal M(k)=0$ for every sufficiently small $k\in K$. The transverse Lyapunov--Schmidt equation therefore has $z=0$ as a solution for every $k$, and uniqueness in the implicit function theorem gives $\psi(k)=0$. Since $d^*k=0$,
$$
\Gamma_{\mathrm{gf}}(k)
=\mathcal E(\Theta+k)-\mathcal E(\Theta)
=\frac12\operatorname{Vol}(\mathbb T^2)\|[\xi,\eta]\|^2.  \qquad \qedhere
$$
\end{proof}

\subsection{The optimal exponent for a compact structure group}

The quartic reduced energy is naturally a norm-square of a moment map. This permits the $\mathrm{SU}(2)$ computation to be extended without using the special three-dimensional Gram identity.

\begin{lemma}\label{lem:YM-moment-map}
Let $V=\mathfrak g\oplus\mathfrak g$ with its product inner product and symplectic form
$$
\omega((\xi_1,\eta_1),(\xi_2,\eta_2))
=\langle\xi_1,\eta_2\rangle-\langle\eta_1,\xi_2\rangle.
$$
For the diagonal adjoint action of $G$ on $V$, the map
$$
\mu(\xi,\eta)=[\xi,\eta]\in\mathfrak g\cong\mathfrak g^*
$$
is a moment map, up to the harmless sign convention for moment maps. Consequently, if $q(\xi,\eta)=\|[\xi,\eta]\|^2$, then there are $C,\varepsilon>0$ such that
$$
q(\xi,\eta)^{3/4}\le C\|Dq(\xi,\eta)\|
$$
whenever $q(\xi,\eta)<\varepsilon$.
\end{lemma}

\begin{proof}
For $\zeta\in\mathfrak g$, the infinitesimal diagonal action is
$\zeta_V(\xi,\eta)=([\zeta,\xi],[\zeta,\eta])$. By invariance of the inner product, for a tangent vector $(u,v)$,
\begin{align*}
\omega(\zeta_V(\xi,\eta),(u,v))
&=\langle[\zeta,\xi],v\rangle-\langle[\zeta,\eta],u\rangle\\
&=\langle\zeta,[\xi,v]+[u,\eta]\rangle\\
&=D\langle[\xi,\eta],\zeta\rangle[(u,v)].
\end{align*}
Thus $\mu(\xi,\eta)=[\xi,\eta]$ is a moment map for this linear compact-group action. The $G$-invariant complex structure $J(\xi,\eta)=(-\eta,\xi)$ makes $V$ a unitary linear representation. Fisher~\cite[Theorem~4.7]{Fisher2014} proves that the norm-square of a moment map associated to a linear action of a compact group satisfies a gradient inequality with exponent $3/4$. Applying that theorem to $q=\|\mu\|^2$ proves the claim.
\end{proof}

\begin{theorem}\label{thm:YM-gauge-fixed}
Let $G$ be a compact Lie group with Lie algebra $\mathfrak g$, let $\Theta$ be the trivial flat connection on $\mathbb T^2$, and let $p>2$.
\begin{enumerate}
\item If $\mathfrak g$ is abelian, then
$$
\theta_{\mathrm{LS}}(\mathcal E_{\mathrm{gf}},0)=\frac12.
$$
\item If $\mathfrak g$ is nonabelian, then there are $C,\varepsilon>0$ such that
$$
|\mathcal E_{\mathrm{gf}}(a)-\mathcal E_{\mathrm{gf}}(0)|^{3/4}
\le C\|\mathcal M(a)\|_{W^{-1,p}}
$$
for $\|a\|_{W^{1,p}}<\varepsilon$, and
$$
\theta_{\mathrm{LS}}(\mathcal E_{\mathrm{gf}},0)=\frac34.
$$
\end{enumerate}
\end{theorem}

\begin{proof}
By Lemmas~\ref{lem:YM-analytic} and~\ref{lem:YM-Hodge}, Theorem~\ref{thm:exact_reduction} applies, and Lemma~\ref{lem:reduction} gives $\Gamma_{\mathrm{gf}}(\xi,\eta)=\kappa\|[\xi,\eta]\|^2$.
If $\mathfrak g$ is abelian, this reduced germ is identically zero. Proposition~\ref{prop:morse_bott_reduction} and Theorem~\ref{thm:exact_reduction} then give the optimal exponent $1/2$.

Suppose $\mathfrak g$ is nonabelian. Then $\Gamma_{\mathrm{gf}}$ is a nonzero homogeneous polynomial of degree four, so Theorem~\ref{thm:order_bound} gives
$$
\theta_R(\Gamma_{\mathrm{gf}},0)\ge\frac34.
$$
On the other hand, Lemma~\ref{lem:YM-moment-map} gives the reduced gradient inequality with exponent $3/4$, so
$\theta_R(\Gamma_{\mathrm{gf}},0)\le3/4$. Therefore
$\theta_R(\Gamma_{\mathrm{gf}},0)=3/4$, and exact reduction gives the same value for $\mathcal E_{\mathrm{gf}}$. Since Theorem~\ref{thm:exact_reduction} identifies the sets of admissible exponents, the displayed full Banach inequality follows as well.
\end{proof}

The same argument gives the ungauged estimate on the linear Coulomb slice.

\begin{corollary}\label{cor:YM-coulomb}
Assume that $\mathfrak g$ is nonabelian and let
$$
\mathcal S=\{a\in W^{1,p}(\mathbb T^2;\Lambda^1\otimes\mathfrak g):d^*a=0\}.
$$
Then there are $C,\varepsilon>0$ such that
$$
|\mathcal E(\Theta+a)-\mathcal E(\Theta)|^{3/4}
\le C\|d_{\Theta+a}^*F_{\Theta+a}\|_{W^{-1,p}}
$$
for $a\in\mathcal S$ with $\|a\|_{W^{1,p}}<\varepsilon$. The exponent $3/4$ is optimal.
\end{corollary}

\begin{proof}
On $\mathcal S$, one has $\mathcal E_{\mathrm{gf}}(a)=\mathcal E(\Theta+a)$ and
$\mathcal M(a)=d_{\Theta+a}^*F_{\Theta+a}$, so the inequality follows from Theorem~\ref{thm:YM-gauge-fixed}. For optimality, choose $k_0=\xi\,dx+\eta\,dy\in K$ with $[\xi,\eta]\ne0$. Then $tk_0\in\mathcal S$ and the exact reduced formula gives
$$
\Gamma_{\mathrm{gf}}(tk_0)=t^4\Gamma_{\mathrm{gf}}(k_0),
\qquad \Gamma_{\mathrm{gf}}(k_0)>0.
$$
Since $\Gamma_{\mathrm{gf}}$ is homogeneous of degree four, Euler's identity gives $D\Gamma_{\mathrm{gf}}(k_0)\ne0$, and therefore
$$
\|D\Gamma_{\mathrm{gf}}(tk_0)\|=|t|^3\|D\Gamma_{\mathrm{gf}}(k_0)\|.
$$
The gradient identity and the finite-dimensional isomorphism $\rho$ then give $\|\mathcal M(tk_0)\|_{W^{-1,p}}\asymp |t|^3$. Hence no exponent smaller than $3/4$ can be admissible.
\end{proof}

If $\mathfrak g$ is nonabelian, the reduced critical set is the commuting-pair variety
$$
\operatorname{Crit}\Gamma_{\mathrm{gf}}
=\{(\xi,\eta):[\xi,\eta]=0\}.
$$
Indeed, the inclusion from right to left is immediate. Conversely, since $\Gamma_{\mathrm{gf}}$ is homogeneous of degree four, Euler's identity implies that $D\Gamma_{\mathrm{gf}}(\xi,\eta)=0$ forces $4\Gamma_{\mathrm{gf}}(\xi,\eta)=0$, hence $[\xi,\eta]=0$. At the origin this critical set contains both $\mathfrak g\times\{0\}$ and $\{0\}\times\mathfrak g$ but is a proper subset of $\mathfrak g\oplus\mathfrak g$. It therefore cannot be a smooth submanifold at the origin: any tangent space there would contain both coordinate subspaces and hence all of $K$, which would make a smooth critical submanifold locally open. Thus the product connection is not Morse--Bott whenever $\mathfrak g$ is nonabelian. In the $\mathrm{SU}(2)$ case this agrees with Feehan~\cite[Corollary~A.10]{FeehanFlat2019}.

Theorem \ref{thm:YM-gauge-fixed} determines the optimal exponent at the product connection on $T^2$ for every compact structure group. For $SU(2)$, however, the special three-dimensional structure of the Lie algebra allows one to go further: in every torus dimension the constant-mode reduced energy becomes a determinantal sum-of-squares singularity. We now exploit this identification not only at the product connection, but also along the nearby stratum of nonzero constant flat connections.

\subsection{SU(2) over higher-dimensional flat tori}

The exact constant-mode reduction persists in every dimension for the trivial $\mathrm{SU}(2)$ bundle over a flat torus, while Theorem~\ref{thm:determinantal-Simon} computes the resulting singularity.

\begin{theorem}\label{thm:YM-higher-torus}
Let $d\ge2$, let $\mathbb T^d=\mathbb R^d/\mathbb Z^d$ carry its standard flat metric, and let $P=\mathbb T^d\times\mathrm{SU}(2)$ be the trivial bundle with product connection $\Theta$. Choose $p>d$, put
$X_d=W^{1,p}(\mathbb T^d;\Lambda^1\otimes\mathfrak{su}(2))$ and
$Y_d=(W^{1,p'}(\mathbb T^d;\Lambda^1\otimes\mathfrak{su}(2)))^*$, and denote the gauge-fixed functional and its gradient map by $\mathcal E_{\mathrm{gf},d}$ and $\mathcal M_d$. Then the Lyapunov--Schmidt correction with respect to the harmonic splitting vanishes identically and, writing a harmonic one-form as $k=\sum_{i=1}^d\xi_i\,dx_i$,
$$
\Gamma_d(\xi_1,\ldots,\xi_d)
=\kappa_d\sum_{1\le i<j\le d}\|[\xi_i,\xi_j]\|^2
$$
for a constant $\kappa_d>0$. Consequently,
$$
\theta_{\mathrm{LS}}(\mathcal E_{\mathrm{gf},d},\Theta)=\frac34.
$$
Moreover, every sufficiently small nonzero constant flat connection $\Theta+k_0$ has optimal exponent $1/2$ for $\mathcal E_{\mathrm{gf},d}$.
\end{theorem}

\begin{proof} The analyticity and Hodge–Fredholm arguments are the same as in the two-dimensional case. The harmonic space is
$$K_d=H^1(\mathbb T^d;\mathfrak{su}(2))\cong\mathfrak{su}(2)^d,$$
and consists precisely of the constant one-forms. Fix the harmonic splitting
$X_d=K_d\oplus Z$
and let $R:Y_d\to\operatorname{Ran}\Delta_H$ be the corresponding transverse projection, so that $\ker R=K_d=Z^\circ$.
Let
$$k=\sum_{i=1}^d\xi_i\,dx_i\in K_d.$$

Then $dk=d^*k=0$ and $F_{\Theta+k}=\sum_{i<j}[\xi_i,\xi_j]\,dx_i\wedge dx_j$.
All coefficients are constant. Hence the gauge-fixed gradient $\mathcal M_d(k)$ is itself a constant $\mathfrak{su}(2)$-valued one-form, and therefore belongs to $K_d$. 
It follows that $R\mathcal M_d(k)=0$
for every sufficiently small $k\in K_d$. Since $z=0$ solves the transverse Lyapunov–Schmidt equation and its small solution is unique, the correction vanishes identically:
$\psi(k)=0$. Thus the reduced energy is exactly the restriction of the gauge-fixed energy to $K_d$:
$$ \Gamma_d(\xi_1,\ldots,\xi_d)  =\kappa_d\sum_{1\le i<j\le d}\|[\xi_i,\xi_j]\|^2 $$
for some $\kappa_d>0$.

\smallskip
Choose an orthonormal identification $\mathfrak{su}(2)\cong\mathbb R^3$. Up to a fixed positive factor, the Lie bracket becomes the cross product. If $X$ is the $3\times d$ matrix with columns $\xi_1,\ldots,\xi_d$, the Cauchy–Binet identity gives
$$\sum_{i<j}\|\xi_i\times\xi_j\|^2=Q_2(X),$$
where $Q_2$ is the sum of the squares of the $2\times2$ minors of $X$. Hence
$\Gamma_d=\kappa_d'Q_2$
for some $\kappa_d'>0$. The determinantal Łojasiewicz theorem, applied with $k=2$ at the rank-zero matrix, yields
$\theta_{\mathrm{LS}}(\Gamma_d,0)=1-\frac{1}{2(2-0)}=\frac34$.
Because the Lyapunov–Schmidt reduction is exact, the full and reduced germs have the same optimal exponent. Therefore
$\theta_{\mathrm{LS}}(\mathcal E_{\mathrm{gf},d},\Theta)=\frac34$.

\smallskip
Now let $k_0=\sum_{i=1}^d\xi_i^0\,dx_i$ 
be a sufficiently small nonzero constant flat connection. Flatness means
$[\xi_i^0,\xi_j^0]=0\qquad\text{for all }i,j$.
Every abelian subalgebra of $\mathfrak{su}(2)$ is one-dimensional. Thus all nonzero $\xi_i^0$ are collinear, and the associated $3\times d$ matrix $X_0$ has rank one.
We now justify the reduction at $k_0$ without assuming that $K_d$ is the kernel of the Hessian there. As $k_0\to0$, we have
$R D\mathcal M_d(k_0)|_Z\longrightarrow\Delta_H|_Z$
in operator norm. Since $\Delta_H|_Z$ is an isomorphism, $R D\mathcal M_d(k_0)|_Z$ remains an isomorphism for all sufficiently small $k_0$.
Write a point near $k_0$ uniquely as $k_0+h+z$, with $h\in K_d$ and $z\in Z$. The implicit-function theorem gives a unique analytic transverse correction $z=\psi_{k_0}(h)$. For every small $h\in K_d$, the form $k_0+h$ is constant, so
$R\mathcal M_d(k_0+h)=0$. Uniqueness therefore implies $\psi_{k_0}(h)=0$. 

\smallskip
The recentered reduced energy is consequently $\widetilde\Gamma_{k_0}(h)=\Gamma_d(k_0+h)-\Gamma_d(k_0)$.
For completeness, the full and reduced optimal exponents still agree at $k_0$. Along the reduction manifold, we have
$\mathcal M_d(k_0+h)\in K_d=Z^\circ$, and restriction gives
$D\widetilde\Gamma_{k_0}(h)=\rho\bigl(\mathcal M_d(k_0+h)\bigr)$,
where $\rho:Z^\circ\to K_d^*$ is an isomorphism. Hence the reduced gradient norm is equivalent to the full gradient norm along $z=0$.
For general small $h$ and $z$, uniform invertibility of
$$\int_0^1R D\mathcal M_d(k_0+h+sz)|_Z\,ds$$
gives the transverse estimate $\|z\|_{X_d}\le C\|\mathcal M_d(k_0+h+z)\|_{Y_d}$.
Using $K_d=Z^\circ$, integrating the energy along the segment in the $z$-direction, and applying the preceding estimate yields
$$\left| \mathcal E_{\mathrm{gf},d}(k_0+h+z) - \mathcal E_{\mathrm{gf},d}(k_0+h) \right| \le C\|\mathcal M_d(k_0+h+z)\|_{Y_d}^2.$$
Lipschitz continuity of $\mathcal M_d$ also gives $\|D\widetilde\Gamma_{k_0}(h)\|\le C\|\mathcal M_d(k_0+h+z)\|_{Y_d}$.

\smallskip
These estimates transfer every reduced Łojasiewicz inequality with exponent $\theta\ge\tfrac12$ to the full functional. Conversely, restriction to $z=0$ transfers every full inequality to the reduced germ. Thus their optimal exponents coincide.
Finally, applying the determinantal theorem to $Q_2$ at the rank-one matrix $X_0$ gives
$$\theta_{\mathrm{LS}}(\widetilde\Gamma_{k_0},0) =1-\frac{1}{2(2-1)}=\frac12.$$
Therefore $\theta_{\mathrm{LS}}(\mathcal E_{\mathrm{gf},d},\Theta+k_0)=\frac12$.
We have completed verifying both assertions. 
\end{proof}

\subsection{The determinantal SU(2) model on the two-torus}
\label{subsec:YM-Simon-pair}

For $d=2$, write $\xi=(x_1,x_2,x_3)$ and $\eta=(y_1,y_2,y_3)$ under an orthonormal identification $\mathfrak{su}(2)\cong\mathbb R^3$. Let
$\mathfrak m=(x_1,x_2,x_3,y_1,y_2,y_3)$ and let $I=(\Delta_{12},\Delta_{13},\Delta_{23})$, where $\Delta_{ij}=x_i y_j-x_j y_i$. Up to a positive constant, the reduced energy is
$q(\xi,\eta)=\|\xi\|^2\|\eta\|^2-\langle\xi,\eta\rangle^2=\Delta_{12}^2+\Delta_{13}^2+\Delta_{23}^2$.

\begin{corollary}\label{cor:YM-determinantal-Simon}
With the notation above,
$$
\overline{(q)}^{\,\mathbb R}=\overline{I^2}^{\,\mathbb R},
\qquad
\overline{J(q)}^{\,\mathbb R}=\overline{\mathfrak m I}^{\,\mathbb R}.
$$
The ordinary order valuation at the vertex computes the optimal ratio $3/4$. At every nonzero rank-one point of the determinantal cone $V_{\mathbb R}(I)$, the local gradient exponent of $q$ is $1/2$.
\end{corollary}

\begin{proof}
Apply Theorem~\ref{thm:determinantal-Simon} to a $3\times2$ matrix with $k=2$. Then $I_2=I$, $I_1=\mathfrak m$, and $Q_2=q$. At the origin, $\operatorname{ord}(q)=4$ and $\operatorname{ord}(J(q))=3$, so the ordinary blowup valuation gives $3/4$. The rank-one statement is the case $s=k-1$ of the same theorem.
\end{proof}

The Yang--Mills application is a case in which geometry identifies the reduced energy exactly, after which the optimal exponent can be read from a concrete algebraic singularity. The next application exhibits a complementary phenomenon. For a resonant semilinear equation the transverse correction is generally nonzero, and the central problem is instead to determine how successive terms of the reduced germ emerge when lower-order obstructions cancel. The semilinear model therefore tests the higher-order content of the intrinsic reduction rather than only its exact-restriction case.

\section{A resonant semilinear Dirichlet energy}\label{sec:semilinear}

\subsection{Setting and main result}

Let $\Omega\subset\mathbb{R}^d$ be a bounded, connected domain with $C^{1,1}$ boundary. We impose homogeneous Dirichlet boundary conditions. Fix an integer $N\ge3$, a real number $\lambda$, and an exponent $p>\max\{d,2\}$, with $p'=p/(p-1)$. Consider the energy of the semilinear model
$$\mathcal{E}_N(u)=\frac12\int_\Omega\bigl(|\nabla u|^2-\lambda u^2\bigr)\,dx
       +\frac1N\int_\Omega u^N\,dx.$$
We use the asymmetric Banach pair $X:=W^{1,p}_0(\Omega)$ and $Y:=\bigl(W^{1,p'}_0(\Omega)\bigr)^*=:W^{-1,p}(\Omega)$.
Since $p>2$, the inclusion $X\hookrightarrow W^{1,p'}_0(\Omega)$ is continuous, so restriction of the duality on $Y\times W^{1,p'}_0$ gives a continuous pairing on 
$Y\times X$. Moreover, $p>d$ gives $X\hookrightarrow L^\infty(\Omega)$.

\smallskip
For $u,v\in X$, we compute $d\mathcal{E}_N(u)[v]=\int_\Omega(\nabla u\cdot\nabla v-\lambda uv+u^{N-1}v)\,dx=\langle\mathcal M_N(u),v\rangle_{Y,X}$, where $\mathcal{M}_N(u):=-\Delta_D u-\lambda u+u^{N-1}\in Y$
and $-\Delta_D:X\to Y$ is the weak Dirichlet realization, defined by $\langle-\Delta_Du,\varphi\rangle=\int_\Omega\nabla u\cdot\nabla\varphi\,dx$, and distinct from the self-adjoint $L^2$ Dirichlet Laplacian. To avoid ambiguity, write $\sigma_D(-\Delta)$ for the spectrum of the self-adjoint Dirichlet Laplacian on $L^2(\Omega)$. 
The map $\mathcal M_N:X\to Y$ is a real-analytic polynomial, and its derivative at the origin is $L:=D\mathcal M_N(0)=-\Delta_D-\lambda I:X\to Y$.

\smallskip
The optimal exponent depends on whether $\lambda$ is a Dirichlet eigenvalue.
\begin{theorem} \label{thm:semilinear_exp}
The following statements hold.
\begin{enumerate}
\item If $\lambda\notin\sigma_D(-\Delta)$, then $\theta_{\mathrm{LS}}(\mathcal E_N,0)=1/2$.
\item Suppose $\lambda\in\sigma_D(-\Delta)$ and let $K=\ker(-\Delta_D-\lambda I)$. If $DP_N(\xi)\ne0$ for every $\xi\in K\setminus\{0\}$, where $P_N(\xi)=\frac1N\int_\Omega\xi^N\,dx$, then
$$
\theta_{\mathrm{LS}}(\mathcal E_N,0)=\frac{N-1}{N}.
$$
In particular, this conclusion holds for every even $N$, regardless of the multiplicity of the eigenvalue.
\item Suppose $\lambda\in\sigma_D(-\Delta)$ and $P_N$ vanishes identically on $K$. Let $Z$ be the $L^2$-orthogonal complement of $K$ in $X$. Then $\xi^{N-1}\in\operatorname{Ran}L$ for every $\xi\in K$, and the homogeneous polynomial
$$
\mathcal Q_{2N-2}(\xi):=-\frac12
\left\langle \xi^{N-1},(L|_Z)^{-1}\xi^{N-1}\right\rangle_{Y,X}
$$
is well-defined on $K$. If $D\mathcal Q_{2N-2}(\xi)\ne0$ for every $\xi\in K\setminus\{0\}$, then
$$
\theta_{\mathrm{LS}}(\mathcal E_N,0)=\frac{2N-3}{2N-2}.
$$
\end{enumerate}
\end{theorem}

The three cases of Theorem \ref{thm:semilinear_exp} correspond to successive levels of degeneracy of the reduced problem. Away from resonance there is no kernel and hence no residual singularity. At resonance, the first possible obstruction is the restriction $P_N$ of the nonlinear energy to the eigenspace $K$; when $P_N$ vanishes identically, the transverse correction generates the next obstruction $Q_{2N-2}$. We now establish the compatible Fredholm splitting and compute the reduced energy to the order needed to make this hierarchy precise.

\subsection{The Fredholm Hessian and compatible splittings}

\begin{lemma} \label{lem:compatible}
Preserving the notation $L=-\Delta_D-\lambda I:X\to Y$ from above, we have

\begin{enumerate}
\item The operator $L$ is symmetric and Fredholm of index zero.
\item Its kernel is the finite-dimensional eigenspace $K:=\ker L=\{\xi\in X:-\Delta_D\xi=\lambda\xi\}$.
\item There are compatible topological splittings $X=K\oplus Z$, $Y=\operatorname{Ran}L\oplus K$, and $\operatorname{Ran}L=K^\circ$, where $Z$ is the $L^2$-orthogonal complement of $K$ in $X$, $K$ is embedded in $Y$ through the $L^2$ pairing, and $K^\circ:=\{T\in Y:\langle T,k\rangle_{Y,X}=0\text{ for every }k\in K\}$.
\end{enumerate}
\end{lemma}

\begin{proof} Let $A=-\Delta_D+I:X\to Y$. The standard $W^{1,p}_0$ theory for the Dirichlet Laplacian on a bounded $C^{1,1}$ domain makes $A$ a topological isomorphism for $1<p<\infty$; see, for example, Grisvard~\cite{Grisvard1985}. The natural inclusion $J:X\to Y$ is compact: Rellich compactness gives a compact inclusion $X\hookrightarrow L^p(\Omega)$, while $L^p(\Omega)\hookrightarrow Y$ continuously. Since
$L=A-(\lambda+1)J$, it follows that $L$ is a compact perturbation of an isomorphism. Thus $L$ is Fredholm of index zero. Its kernel consists exactly of the $L^2$ Dirichlet eigenfunctions at $\lambda$; elliptic regularity places these eigenfunctions in $X$, and hence $K$ is finite dimensional.

\smallskip
For $u\in X$, the distribution $Lu\in Y$ is defined weakly by $\langle Lu,\varphi\rangle=\int_\Omega(\nabla u\cdot\nabla\varphi-\lambda u\varphi)\,dx$ for $\varphi\in W^{1,p'}_0(\Omega)$. Taking $\varphi=v\in X$ gives $\langle Lu,v\rangle_{Y,X}=\int_\Omega(\nabla u\cdot\nabla v-\lambda uv)\,dx$.
The right-hand side is unchanged when $u$ and $v$ are interchanged. Hence $\langle Lu,v\rangle_{Y,X}=\langle Lv,u\rangle_{Y,X}$.
Thus $L$ is symmetric with respect to the specified pairing. 

\smallskip
Choose an $L^2$-orthonormal basis $e_1,\dots,e_r$ of $K$ and define $P_Xu:=\sum_{j=1}^r(u,e_j)_{L^2}e_j$ and $Z:=\ker P_X$. This yields $X=K\oplus Z$. Define similarly $P_YT:=\sum_{j=1}^r\langle T,e_j\rangle_{Y,X}e_j$.
Symmetry gives $\operatorname{Ran}L\subseteq\ker P_Y$. Since $L$ is Fredholm of index zero, $\operatorname{codim}\operatorname{Ran}L=\dim K=r$, while $P_Y:Y\to K$ is onto and hence $\operatorname{codim}\ker P_Y=r$. Therefore $\operatorname{Ran}L=\ker P_Y$ and $Y=\operatorname{Ran}L\oplus K$. Symmetry also gives $\operatorname{Ran}L\subseteq K^\circ$; again both closed subspaces have codimension $r$, so $\operatorname{Ran}L=K^\circ$.

The restriction map $\rho:K\to K^*$, $\rho(k)(h)=(k,h)_{L^2}$, is the finite-dimensional Riesz isomorphism. To see that $K=Z^\circ$, note first that every $k\in K$ annihilates $Z$ by $L^2$ orthogonality. Conversely, if $T\in Z^\circ$, then for $u=P_Xu+z\in X$ one has $\langle T,u\rangle=\sum_j(u,e_j)_{L^2}\langle T,e_j\rangle$. Thus $T$ agrees on $X$ with the element $\sum_j\langle T,e_j\rangle e_j\in K$; since $X=W^{1,p}_0(\Omega)$ is dense in $W^{1,p'}_0(\Omega)$, the two elements of $Y$ are equal. Finally, $L|_Z:Z\to\operatorname{Ran}L$ is an isomorphism.
\end{proof}

\subsection{Lyapunov--Schmidt reduction}

Let $R:=I-P_Y:Y\to \operatorname{Ran}  L$. For $\xi\in K$ and $z\in Z$, define $F(\xi,z):=R \mathcal{M}_N(\xi+z)$.
Since $D_zF(0,0)=L|_Z:Z\to\operatorname{Ran}L$ is an isomorphism, the analytic implicit-function theorem gives neighborhoods of the origin and a unique analytic map
$\psi:K\supset V\longrightarrow Z $
satisfying $\psi(0)=0, \, D\psi(0)=0, \, R \mathcal{M}_N(\xi+\psi(\xi))=0$.

\begin{definition}
The reduced energy is the finite-dimensional analytic function
$$\Gamma(\xi):=\mathcal{E}_N(\xi+\psi(\xi)), \qquad \xi\in V\subset K. $$
\end{definition}

It is the effective energy obtained after solving the Euler--Lagrange equation in every transverse direction. The exact-reduction theorem gives $\theta_{\mathrm{LS}}(\mathcal E_N,0)=\theta_R(\Gamma,0)$; in fact, the two sets of admissible exponents coincide.

\smallskip
The next result considers the leading terms of the reduction. Fix any norm on the finite-dimensional space $K$ and its associated dual norm on $K^*$. The asymptotic orders below are independent of this choice, although the implied constants may change.

\begin{lemma}  \label{lem:leading}
Let $A:=(L|_Z)^{-1}:\operatorname{Ran}L\to Z$ and $g(\xi):=R(\xi^{N-1})$. As $\xi\to0$ in $K$, one has
\begin{align*}
\psi(\xi)
&=-A g(\xi)+O(\|\xi\|^{2N-3}),\\
\Gamma(\xi)
&=P_N(\xi)-\frac12\langle g(\xi),A g(\xi)\rangle_{Y,X}
  +O(\|\xi\|^{3N-4}),\\
D\Gamma(\xi)
&=DP_N(\xi)
  +D\!\left[-\frac12\langle g(\xi),A g(\xi)\rangle_{Y,X}\right]
  +O(\|\xi\|^{3N-5}),
\end{align*}
where $P_N(\xi):=\frac1N\int_\Omega\xi^N\,dx$. In particular, the second displayed term in $\Gamma$ is homogeneous of degree $2N-2$.
\end{lemma}

\begin{proof}
The transverse equation is $L\psi(\xi)+R(\xi+\psi(\xi))^{N-1}=0$.
Since $D\psi(0)=0$, after shrinking the neighborhood we may assume $\|\psi(\xi)\|\le\|\xi\|$. Applying $(L|_Z)^{-1}$ to the transverse equation and using the continuity of the polynomial map $u\mapsto u^{N-1}$ from $X$ to $Y$ then gives
$\psi(\xi)=O(\|\xi\|^{N-1})$. Hence $(\xi+\psi)^{N-1}=\xi^{N-1}+O(\|\xi\|^{N-2}\|\psi\|+\|\psi\|^{N-1})=\xi^{N-1}+O(\|\xi\|^{2N-3})$, and applying $A$ gives the asserted expansion for $\psi$.
Write $\psi_1=-Ag$. Since $L\xi=0$ and $L$ is symmetric, $\Gamma(\xi)=\frac12\langle L\psi,\psi\rangle_{Y,X}+\frac1N\int_\Omega(\xi+\psi)^N\,dx$.
The error $\psi-\psi_1$ has order $2N-3$, and $L\psi_1=-g$. Therefore
\begin{align*}
\frac12\langle L\psi,\psi\rangle_{Y,X}
&=\frac12\langle g,Ag\rangle_{Y,X}+O(\|\xi\|^{3N-4}),\\
\frac1N\int_\Omega(\xi+\psi)^N\,dx
&=P_N(\xi)+\langle\xi^{N-1},\psi\rangle_{Y,X}+O(\|\xi\|^{3N-4}).
\end{align*}
Since $\xi^{N-1}=P_Y(\xi^{N-1})+g$ with $P_Y(\xi^{N-1})\in K=Z^\circ$ and $Ag\in Z$, one has
$\langle \xi^{N-1},Ag\rangle_{Y,X}=\langle g,Ag\rangle_{Y,X}$. Thus $\langle\xi^{N-1},\psi\rangle_{Y,X}=-\langle g,Ag\rangle_{Y,X}+O(\|\xi\|^{3N-4})$, and adding the two expansions gives the stated formula for $\Gamma$.
The expansion takes place on the finite-dimensional space $K$ and all terms are analytic, so differentiation lowers the order of the remainder by one and yields the formula for $D\Gamma$.
\end{proof}

\begin{proof}[Proof of Theorem~\ref{thm:semilinear_exp}]
If $\lambda\notin\sigma_D(-\Delta)$, then $K=\{0\}$ and $L$ is an isomorphism. The inverse-function theorem applied to $\mathcal M_N$ gives $\|u\|_X\le C\|\mathcal M_N(u)\|_Y$ near zero, while Taylor's theorem gives $|\mathcal E_N(u)|\le C\|u\|_X^2$. Thus $1/2$ is admissible, and it is optimal by definition.

Suppose next that $\lambda\in\sigma_D(-\Delta)$ and $DP_N(\xi)\ne0$ for every nonzero $\xi\in K$. Since $DP_N$ is homogeneous of degree $N-1$ and the unit sphere of $K$ is compact, $\|DP_N(\xi)\|_{K^*}\ge c\|\xi\|^{N-1}$ for some $c>0$. Lemma~\ref{lem:leading} then gives, after shrinking the neighborhood, $|\Gamma(\xi)|\le C\|\xi\|^N$ and $\|D\Gamma(\xi)\|_{K^*}\ge c'\|\xi\|^{N-1}$. Hence $(N-1)/N$ is admissible. Since $P_N$ is not identically zero, choose $a\in K$ with $P_N(a)\ne0$. Along $\xi=ta$ one has $\Gamma(ta)=P_N(a)t^N+O(|t|^{2N-2})$ and $\|D\Gamma(ta)\|_{K^*}=O(|t|^{N-1})$, so any admissible exponent satisfies $N\theta\ge N-1$. Thus $\theta_{\mathrm{LS}}=(N-1)/N$. If $N$ is even, $P_N(\xi)=N^{-1}\int_\Omega|\xi|^Ndx>0$ for $\xi\ne0$, and Euler's identity gives $DP_N(\xi)[\xi]=NP_N(\xi)\ne0$.

Finally assume $P_N\equiv0$ on $K$. Then $DP_N\equiv0$, and for every $\xi,h\in K$,
$0=DP_N(\xi)[h]=\langle\xi^{N-1},h\rangle_{Y,X}$. Thus $P_Y(\xi^{N-1})=0$ and $\xi^{N-1}\in\operatorname{Ran}L$. In Lemma~\ref{lem:leading} we therefore have $g(\xi)=\xi^{N-1}$, so
$$
\Gamma(\xi)=\mathcal Q_{2N-2}(\xi)+O(\|\xi\|^{3N-4}),
\qquad
D\Gamma(\xi)=D\mathcal Q_{2N-2}(\xi)+O(\|\xi\|^{3N-5}).
$$
The hypothesis on $D\mathcal Q_{2N-2}$ and compactness of the unit sphere give $\|D\mathcal Q_{2N-2}(\xi)\|\ge c\|\xi\|^{2N-3}$. After shrinking the neighborhood, the displayed expansions therefore imply
$|\Gamma(\xi)|\le C\|\xi\|^{2N-2}$ and $\|D\Gamma(\xi)\|\ge c'\|\xi\|^{2N-3}$. Consequently,
$\|D\Gamma(\xi)\|\ge C'|\Gamma(\xi)|^{(2N-3)/(2N-2)}$, so $(2N-3)/(2N-2)$ is admissible. Since $\mathcal Q_{2N-2}$ is not identically zero, choose $a$ with $\mathcal Q_{2N-2}(a)\ne0$ and argue along $ta$ exactly as above to obtain the reverse inequality. Exact reduction completes the proof.
\end{proof}

For even $N$, Theorem~\ref{thm:semilinear_exp}(2) recovers the exponent in Haraux--Jendoubi--Kavian~\cite[Theorem~2.2]{HarauxJendoubiKavian2003}; their Theorem~2.5 also treats the simple-kernel cubic noncancellation case. The next theorem shows that, for a simple resonance, all possible higher cancellations occur in a rigid arithmetic progression of orders.

Theorem \ref{thm:semilinear_exp} stops as soon as either the first or the second obstruction is nonzero. When the resonance is simple, the one-dimensional reduced problem has enough additional structure to continue this process to every order. Rather than introducing successive obstruction coefficients one at a time, the next result packages all possible cancellations into a single analytic germ.

\begin{theorem}\label{thm:obstruction-ladder}
Assume that $\lambda$ is a simple Dirichlet eigenvalue and $K=\mathbb Re$, where $\|e\|_{L^2}=1$. There is a real analytic germ $H:(\mathbb R,0)\to\mathbb R$ such that
$$
\Gamma(te)=t^2H(t^{N-2}).
$$
Writing $H(s)=\sum_{r\ge1}c_rs^r$, one has
$$
c_1=\frac1N\int_\Omega e^N\,dx.
$$
If $c_1=0$, then $e^{N-1}\in\operatorname{Ran}L$ and
$$
c_2=-\frac12
\left\langle e^{N-1},(L|_Z)^{-1}e^{N-1}\right\rangle_{Y,X}.
$$
If $H\equiv0$, then $\Gamma\equiv0$, the functional is Morse--Bott at the origin, and $\theta_{\mathrm{LS}}=1/2$. Otherwise, if $r_*$ is the least index with $c_{r_*}\ne0$, then
$$
\operatorname{ord}_0\Gamma=2+r_*(N-2),
\qquad
\theta_{\mathrm{LS}}(\mathcal E_N,0)
=1-\frac{1}{2+r_*(N-2)}.
$$
Thus the possible non-Morse--Bott exponents for a simple resonance belong to the discrete ladder
$$
1-\frac1N,\quad
1-\frac1{2N-2},\quad
1-\frac1{3N-4},\quad
1-\frac1{4N-6},\ \ldots .
$$
\end{theorem}

\begin{proof}
Let $A=(L|_Z)^{-1}$ and consider the analytic equation on $\mathbb R\times Z$,
$$
\Phi(s,\chi):=L\chi+sR(e+\chi)^{N-1}=0.
$$
Since $D_\chi\Phi(0,0)=L|_Z$ is an isomorphism, the analytic implicit-function theorem gives a unique analytic map $\chi(s)$ with $\chi(0)=0$ and $\Phi(s,\chi(s))=0$. For $s=t^{N-2}$, the element $t\chi(s)$ solves the original transverse equation at $te$, because
$L(t\chi)+R(t(e+\chi))^{N-1}=t\Phi(t^{N-2},\chi)$. Uniqueness of the Lyapunov--Schmidt correction therefore gives
$\psi(te)=t\chi(t^{N-2})$.

Since $L e=0$, substitution into the energy yields
$$
\Gamma(te)=t^2H(t^{N-2}),
\qquad
H(s):=\frac12\langle L\chi(s),\chi(s)\rangle_{Y,X}
+\frac{s}{N}\int_\Omega(e+\chi(s))^N\,dx.
$$
This $H$ is analytic and $H(0)=0$. Because $\chi(s)=O(s)$, its linear coefficient is $c_1=N^{-1}\int e^N$. If $c_1=0$, then $e^{N-1}\perp K$, so $R(e^{N-1})=e^{N-1}$, and differentiating $\Phi(s,\chi(s))=0$ at $s=0$ gives $\chi'(0)=-Ae^{N-1}$. Expanding the two terms in $H$ to order $s^2$ gives
$$
\frac12\langle L\chi,\chi\rangle
=\frac{s^2}{2}\langle e^{N-1},Ae^{N-1}\rangle+O(s^3),
$$
while
$$
\frac{s}{N}\int(e+\chi)^N
=-s^2\langle e^{N-1},Ae^{N-1}\rangle+O(s^3).
$$
Hence the displayed formula for $c_2$ follows.

If $H\equiv0$, then $\Gamma\equiv0$ on the one-dimensional kernel and Proposition~\ref{prop:morse_bott_reduction} applies. Otherwise, if $c_{r_*}$ is the first nonzero coefficient, then
$\Gamma(te)=c_{r_*}t^{2+r_*(N-2)}+\text{higher-order terms}$. A nonzero one-variable analytic germ of order $d$ has optimal gradient exponent $(d-1)/d$, so exact reduction gives the stated formula.
\end{proof}

The first two rungs recover the alternatives previously visible from Lemma~\ref{lem:leading}: if $A_N(e):=\int_\Omega e^N\,dx\ne0$, then the exponent is $(N-1)/N$; if $A_N(e)=0$ but
$B_N(e):=\langle e^{N-1},(L|_Z)^{-1}e^{N-1}\rangle\ne0$, then it is $(2N-3)/(2N-2)$.

The abstract obstruction ladder becomes especially transparent on an interval. There every resonant eigenspace is one dimensional, and the first two obstruction coefficients can be computed explicitly. For odd $N$, the parity of the eigenmode determines whether the leading coefficient vanishes, while a Green-kernel argument shows that in the cancelling case the second coefficient is always nonzero. This yields a complete classification of the optimal exponent.

\begin{theorem}\label{thm:interval-classification}
Let $\Omega=(0,\pi)$ and let $e_k(x)=\sqrt{2/\pi}\sin(kx)$. For the energy $\mathcal E_N$ one has
$$
\theta_{\mathrm{LS}}(\mathcal E_N,0)=
\begin{cases}
\dfrac12, & \lambda\notin\{1^2,2^2,3^2,\ldots\},\\[5pt]
\dfrac{N-1}{N}, & \lambda=k^2\text{ and }N\text{ is even},\\[5pt]
\dfrac{N-1}{N}, & \lambda=k^2,\ N\text{ is odd, and }k\text{ is odd},\\[5pt]
\dfrac{2N-3}{2N-2}, & \lambda=k^2,\ N\text{ is odd, and }k\text{ is even}.
\end{cases}
$$
More precisely, in the last case
$$
B_N(e_k)=-\frac{C_N}{k^2}<0,
$$
where $C_N>0$ depends only on $N$.
\end{theorem}

\begin{proof}
The nonresonant case and the case of even $N$ follow from Theorem~\ref{thm:semilinear_exp}. Assume that $N$ is odd and $\lambda=k^2$. Put $c=(2/\pi)^{1/2}$ and $S_N=\int_0^\pi\sin^N y\,dy>0$. Then
$$
A_N(e_k)=\int_0^\pi e_k^N\,dx
=\frac{c^N S_N}{k}\sum_{j=0}^{k-1}(-1)^j.
$$
Hence $A_N(e_k)\ne0$ exactly when $k$ is odd, and Theorem~\ref{thm:obstruction-ladder} gives the third line of the formula.

Suppose now that $k$ is even. Set $f(y)=\sin^{N-1}y$, which is nonnegative, nonzero, and $\pi$-periodic. On the circle $\mathbb R/\pi\mathbb Z$, the operator $\mathcal L=-d^2/dy^2-1$ is invertible, since its Fourier eigenvalues are $4j^2-1$, $j\in\mathbb Z$. Its periodic Green kernel is
$$
G(y,z)=-\frac12\sin\delta(y,z),
$$
where $\delta(y,z)\in[0,\pi)$ is the representative of $y-z$ modulo $\pi$. Indeed, away from the diagonal this solves the homogeneous equation, it is periodic and continuous, and the derivative jump is the one required by $\mathcal LG=\delta_z$. In particular, $G(y,z)<0$ off the diagonal. Therefore the unique periodic solution
$U(y)=\int_0^\pi G(y,z)f(z)\,dz$ of $\mathcal LU=f$ satisfies $U(y)<0$ for every $y$.

Let $L_k=-d^2/dx^2-k^2$ and define
$$
\widehat w_k(x)=\frac{c^{N-1}}{k^2}\bigl(U(kx)-U(0)\cos(kx)\bigr).
$$
Because $k$ is even and $U$ is $\pi$-periodic, $\widehat w_k$ satisfies the Dirichlet boundary conditions, and $L_k\widehat w_k=e_k^{N-1}$. The $Z$-solution $(L_k|_Z)^{-1}e_k^{N-1}$ differs from $\widehat w_k$ by a multiple of $e_k$. Since $A_N(e_k)=0$, this difference does not affect its pairing with $e_k^{N-1}$. The cosine correction also contributes nothing because $\int_0^{k\pi}f(y)\cos y\,dy=0$. Hence
$$
B_N(e_k)
=\frac{c^{2N-2}}{k^3}\int_0^{k\pi}f(y)U(y)\,dy
=\frac{c^{2N-2}}{k^2}\int_0^\pi f(y)U(y)\,dy.
$$
The last integral is strictly negative, so $B_N(e_k)=-C_N/k^2$ with
$C_N:=-c^{2N-2}\int_0^\pi fU>0$. The second rung of Theorem~\ref{thm:obstruction-ladder} now gives the final exponent.
\end{proof}

\begin{example}\label{ex:cubic-cancellation}
For $N=3$, the periodic solution in the preceding proof is $U(y)=-\frac12-\frac16\cos(2y)$, and one obtains $C_3=5/(6\pi)$. Thus, for every even $k$,
$$
B_3(e_k)=-\frac{5}{6\pi k^2},
\qquad
\theta_{\mathrm{LS}}(\mathcal E_3,0)=\frac34.
$$
At $k=2$ this gives $B_3(e_2)=-5/(24\pi)$ and
$\Gamma(te_2)=\frac{5}{48\pi}t^4+O(|t|^5)$, recovering the explicit calculation that motivated the cancellation theorem.
\end{example}

For higher-dimensional resonant eigenspaces, singular directions of $P_N$ may require further terms of the reduced germ. The intrinsic Simon-pair formula of Section~\ref{sec:intrinsic} then supplies the exact finite-dimensional procedure: expand $\Gamma$ far enough to resolve those directions and compute the resulting arc or divisorial ratios.

Taken together, the two applications exhibit complementary manifestations of the intrinsic Simon singularity. In the Yang--Mills problem it records a geometric rank stratification through a determinantal model, whereas in the semilinear problem it records a hierarchy of analytic resonance obstructions. In both cases, the optimal exponent of the original infinite-dimensional functional is governed by finite-dimensional singularity data, which can then be approached by algebraic, valuative, or explicit analytic methods.

\end{document}